\documentclass[12pt,a4paper,reqno]{article}
\usepackage{amsmath,amsfonts}
\usepackage{amsthm}
\usepackage{upref}
\usepackage{enumerate}
\usepackage{verbatim}
\usepackage{tikz}
\usepackage{authblk}
\usepackage[margin=1in]{geometry}
\usepackage[%
pdfauthor={Daniel Daners, Jochen Glueck, James B Kennedy},%
pdftitle={Eventually and asymptotically positive semigroups on
  Banach lattices}%
]{hyperref}

\title{Eventually and asymptotically positive semigroups on
  Banach lattices}

\author[1]{Daniel Daners}%
\author[2]{Jochen Gl\"uck\thanks{Supported by a scholarship within the
    scope of the LGFG Baden-W\"urttemberg, Germany.}}%
\author[3]{James B. Kennedy\thanks{Partly supported by a fellowship of
    the Alexander von Humboldt Foundation, Germany.}}%
\affil[1]{School of Mathematics and Statistics, University of Sydney,
  NSW 2006, Australia\authorcr%
  \nolinkurl{daniel.daners@sydney.edu.au}}%
\affil[2]{Institut f\"ur Angewandte Analysis, Universit\"at Ulm,
  D-89069 Ulm, Germany\authorcr%
  \nolinkurl{jochen.glueck@uni-ulm.de}}%
\affil[3]{Institut f\"ur Analysis, Dynamik und Modellierung,
  Universit\"at Stuttgart, Pfaffenwaldring~57, D-70659 Stuttgart,
  Germany\authorcr%
  \nolinkurl{james.kennedy@mathematik.uni-stuttgart.de}}%

\date{\today}

\newtheorem{theorem}{Theorem}[section]
\newtheorem{lemma}[theorem]{Lemma}
\newtheorem{proposition}[theorem]{Proposition}
\newtheorem{corollary}[theorem]{Corollary}

\theoremstyle{definition}
\newtheorem{definition}[theorem]{Definition}

\newtheorem{example}[theorem]{Example}
\newtheorem{examples}[theorem]{Examples}

\theoremstyle{remark}
\newtheorem{remark}[theorem]{Remark}
\newtheorem{remarks}[theorem]{Remarks}

\numberwithin{equation}{section}

\DeclareMathOperator{\id}{\mathit I}
\DeclareMathOperator{\one}{\mathbf 1}
\DeclareMathOperator{\im}{im}
\DeclareMathOperator{\spr}{r}
\DeclareMathOperator{\spb}{s}
\DeclareMathOperator{\trace}{\gamma}

\DeclareMathOperator{\repart}{Re}

\DeclareMathOperator{\per}{per}
\DeclareMathOperator{\dist}{dist}

\newcommand{\bbN}{\mathbb{N}}
\newcommand{\bbZ}{\mathbb{Z}}
\newcommand{\bbR}{\mathbb{R}}

\newcommand{\bbC}{\mathbb{C}}
\newcommand{\bbT}{\mathbb{T}}
\newcommand{\calL}{\mathcal{L}}

\newcommand{\distPos}[1]{\operatorname{d}_+\left(#1\right)}

\newcommand{\phdot}{\mathord{\,\cdot\,}}

\let\oldthebibliography\thebibliography
\renewcommand\thebibliography[1]{
  \oldthebibliography{#1}
  \setlength{\parskip}{0pt}
  \setlength{\itemsep}{0pt plus 0.3ex}
  \small
}

\begin{document}

\maketitle

\begin{abstract}
  We develop a theory of eventually positive $C_0$-semigroups on Banach
  lattices, that is, of semigroups for which, for every positive initial
  value, the solution of the corresponding Cauchy problem becomes
  positive for large times. We give characterisations of such semigroups
  by means of spectral and resolvent properties of the corresponding
  generators, complementing existing results on spaces of continuous
  functions. This enables us to treat a range of new examples including
  the square of the Laplacian with Dirichlet boundary conditions, the
  bi-Laplacian on $L^p$-spaces, the Dirichlet-to-Neumann operator on
  $L^2$ and the Laplacian with non-local boundary conditions on $L^2$
  within the one unified theory.  We also introduce and analyse a weaker
  notion of eventual positivity which we call ``asymptotic positivity'',
  where trajectories associated with positive initial data converge to
  the positive cone in the Banach lattice as $t \to \infty$. This allows
  us to discuss further examples which do not fall within the
  above-mentioned framework, among them a network flow with non-positive
  mass transition and a certain delay differential equation.
\end{abstract}

{\footnotesize\textbf{Mathematics Subject Classification (2010):}
47D06, 47B65, 34G10, 35B09, 47A10}%
\par\noindent {\footnotesize\textbf{Keywords:} One-parameter semigroups
  of linear operators; semigroups on Banach lattices; eventually
  positive semigroup; asymptotically positive semigroup; positive
  spectral projection; eventually positive resolvent; asymptotically
  positive resolvent; Perron-Frobenius theory}

\section{Introduction}
\label{section:introduction-preliminaries}
While the study of positive operator semigroups is by now a classical
topic in the theory of $C_0$-semigroups (see e.g.~\cite{Arendt1986} for
a survey), the analysis of \emph{eventually} positive semigroups,
i.e.~semigroups which only become positive for positive, possibly large,
times, in various contexts seems to have emerged only within the last
decade.  Probably the first example, the idea to consider matrices whose 
powers are eventually positive, is somewhat older and was in large part 
motivated for example by the consideration of inverse eigenvalues problems 
(see e.g.~\cite{Brauer1961,Friedland1978,Zaslavsky1999} or
\cite[pp.~48--54]{Seneta1981}) and by an attempt to generalise the
classical Perron--Frobenius type spectral results to a wider class of
matrices (see e.g.~\cite{Elhashash2008,MR2117663}).

An analysis of continuous-time eventually positive matrix semigroups can
for example be found in \cite{Noutsos2008}; see also \cite{Olesky2009}
for some related results. The phenomenon of eventually
positive solutions of Cauchy problems was also observed at around the
same time in an infinite-dimensional setting in the context of
biharmonic equations; see \cite{FGG08} and \cite{GaGr08}. Another
infinite-dimensional occurrence of eventual positivity was analysed in
\cite{Daners2014}, where it was proven that the semigroup generated by
a class of Dirichlet-to-Neumann operators on a disk is eventually positive but
not positive in some cases.  A first attempt to develop a unified theory
of eventually positive $C_0$-semigroups was subsequently made by the
current authors in \cite{DanersI}, providing some spectral results on
Banach lattices and a characterisation of eventually strongly positive
semigroups on $C(K)$-spaces with $K$ compact.

The current paper has two principal aims. The first aim is to
characterise eventual strong positivity of resolvents and
$C_0$-semigroups on \emph{general} Banach lattices, not just in $C(K)$; see
Sections~\ref{section:projections-strong}--\ref{section:semigroups-strong}.
Unlike in $C(K)$-spaces, the positive cone in general Banach lattices may have
empty interior, a fact which poses
new challenges but allows us to consider a wide variety of new
examples: on Hilbert lattices, on $L^p$-spaces and on spaces of continuous
functions vanishing at the boundary of a sufficiently smooth bounded
domain; see Section~\ref{section:applications-strong}.

The second aim is to cover situations where the $C_0$-semigroup does not
satisfy the assumptions made in
Sections~\ref{section:projections-strong}--\ref{section:semigroups-strong},
but where there is nevertheless some weaker form of ``eventual
positivity''. This is done in
Sections~\ref{section:resolvents-asymptotic} and
\ref{section:semigroups-asymptotic}, where we introduce and characterise
a notion which we call \emph{asymptotic positivity}, where, roughly
speaking, denoting our semigroup by $(e^{tA})_{t \ge 0}$, the distance
of $e^{tA}f$ to the positive cone of the Banach lattice converges to
zero as $t \to \infty$, whenever $f$ itself is in the cone.  In this
framework we are able to drop the distinction between individual and
uniform eventual behaviour, which was necessary in our theory on
eventual positivity so far.
We give a number of applications in Section~\ref{section:applications-asymptotic}: 
we revisit the finite-dimensional case, a bi-Laplacian, and the case of self-adjoint 
operators on Hilbert lattices, including Dirichlet-to-Neumann operators; we also 
give a couple of new examples, namely a transport problem on a metric graph 
and a particular delay differential equation which generate semigroups that are 
asymptotically positive but not positive nor eventually positive in any 
previously introduced sense.

In the final section, Section~\ref{section:open-problems}, we collect a number 
of problems which are left open in, or emerge from, the current paper.

We note in passing that there are other notions of eventual or asymptotic positivity, 
for example in \cite{Chung2015}, where a positive forcing term is introduced 
to obtain a form of asymptotic positivity. We shall investigate eventual 
positivity as an inherent property of the semigroup, in particular without any
such forcing term.

Let us now formulate two theorems giving a -- somewhat incomplete --
summary of our main results, the first on eventual and the second on
asymptotic positivity. In what follows, we will denote by $E_+$ the
positive cone of a Banach lattice $E$; if $u \in E_+$, then $E_u$ is the
principal ideal generated by $u$, and we will write $v \gg_u 0$ and say
$v$ is strongly positive with respect to $u$ if there is a $c>0$ such
that $v \geq cu$. We refer to Section~\ref{section:notation} for a
complete description of our notation, and to
Definitions~\ref{def:resolvents-strong} and \ref{def:semigroups-strong}
for the relevant terminology.

\begin{theorem}
  Let $(e^{tA})_{t \ge 0}$ be a real and eventually differentiable
  $C_0$-semigroup with $\sigma(A)\neq\emptyset$ on a complex Banach
  lattice $E$. Suppose that the peripheral spectrum $\sigma_{\per}(A)$
  is finite and consists of poles of the resolvent. If $u \in E_+$ is
  such that $D(A) \subseteq E_u$, then the following assertions are
  equivalent:
  \begin{enumerate}[\upshape (i)]
  \item The semigroup $(e^{tA})_{t \ge 0}$ is individually eventually
    strongly positive with respect to $u$.
  \item The spectral bound $\spb(A)$ of $A$ is a dominant spectral value and the
    resolvent $R(\phdot,A)$ is individually eventually strongly positive
    with respect to $u$ at $\spb(A)$.
  \item $\spb(A)$ is a dominant spectral value and the spectral projection
    $P$ associated with $\spb(A)$ fulfils $Pf \gg_u 0$ for every $f>0$.
  \item $\spb(A)$ is a dominant spectral value. Moreover,
    $\ker(\spb(A)\id-A)$ is spanned by a vector $v \gg_u 0$ and
    $\ker(\spb(A)\id-A')$ contains a strictly positive functional.
  \end{enumerate}
\end{theorem}

The assertions of the above theorem are shown in
Corollary~\ref{cor:projections-strong},
Theorem~\ref{thm:resolvents-strong} and
Corollary~\ref{cor:semigroups-strong-differentiable}.
 
The eventual differentiability of the semigroup and the domination
condition $D(A) \subseteq E_u$ can be partially weakened at the expense
of losing the equivalent assertion (ii) and of needing an additional
boundedness condition in assertions (iii) and (iv); see
Theorem~\ref{thm:semigroups-strong} and
Corollary~\ref{cor:projections-strong}.

For the precise definition of \emph{asymptotic positivity} which appears
in the next theorem we refer the reader to
Definitions~\ref{def:resolvents-asymptotic} and
\ref{def:semigroups-asymptotic}.

\begin{theorem}
  Let $(e^{tA})_{t\ge 0}$ be an eventually norm continuous
  $C_0$-semigroup with $\sigma(A)\neq\emptyset$ on a complex Banach
  lattice $E$. Suppose that the peripheral spectrum $\sigma_{\per}(A)$
  is finite and consists of simple poles of the resolvent. Then the
  following assertions are equivalent:
  \begin{itemize}
  \item[\upshape (i)] The semigroup $(e^{tA})_{t \ge 0}$ is individually
    asymptotically positive.
  \item[\upshape (i')] The semigroup $(e^{tA})_{t \ge 0}$ is uniformly
    asymptotically positive.
  \item[\upshape (ii)] $\spb(A)$ is a dominant spectral value of $A$ and
    $R(\phdot,A)$ is individually asymptotically positive.
  \item[\upshape (ii')] $\spb(A)$ is a dominant spectral value of $A$
    and $R(\phdot,A)$ is uniformly asymptotically positive.
  \item[\upshape (iii)] $\spb(A)$ is a dominant spectral value of $A$
    and the associated spectral projection $P$ is positive.
  \end{itemize}
\end{theorem}

This theorem follows from Theorems~\ref{thm:resolvents-asymptotic} and
\ref{thm:semigroups-asymptotic} and from 
Remark~\ref{rem:eventually-norm-continuous-semigroup}.

\section{Notation and preliminaries}
\label{section:notation}
Throughout this article, we will use the following notation.  We assume
that the reader is familiar with the theory of $C_0$-semigroups (see for instance
\cite{Engel2000,Engel2006}), and with the theory of real and complex Banach
lattices (see for instance \cite{Schaefer1974,Meyer-Nieberg1991}).
If $E$ is a complex Banach lattice, then it is by definition
the complexification of a real Banach lattice and we always denote
this real Banach lattice by $E_\bbR$ and call it the \emph{real part}
of $E$. Throughout, we suppose that $E$ and $F$
are Banach spaces and denote by $\calL(E,F)$ the space of bounded linear
operators from $E$ to $F$ (or by $\calL(E)$ if $F = E$); in the special
case where $E\subseteq F$ and the natural embedding is continuous we
write $E\hookrightarrow F$.

\paragraph{Positivity and related notions}
Suppose $E$ and $F$ be real or complex Banach lattices. We denote by
\begin{displaymath}
  E_+ = \{u \in E\colon u \ge 0\}
\end{displaymath}
the \emph{positive cone} in $E$. An element $u \in E_+$ is called
\emph{positive}. We write $u > 0$ to say that $u \ge 0$ and $u \neq 0$.
For $f \in E$ we denote by
\begin{equation}
  \label{eq:def-distPos}
  \distPos{f} := \dist(f,E_+)
\end{equation}
the distance of $f$ to the positive cone $E_+$.  As usual, the
\emph{principal ideal} generated by $u \in E_+$ is given by
\begin{displaymath}
  E_u :=\bigl\{f \in E\colon\text{$\exists c \ge 0$, $|f| \le c u$}\bigr\}.
\end{displaymath}
If $E$ is a real Banach lattice, we define the \emph{gauge norm} of $f
\in E_u$ with respect to $u$ by
\begin{displaymath}
  \|f\|_u := \inf\{\lambda\geq 0\colon |f|\leq\lambda u\}.
\end{displaymath}
If $E$ is a complex Banach lattice, and thus the complexification of a
real Banach lattice $E_{\mathbb R}$, then we define the gauge norm
$\|\phdot\|_u$ on $E_u$ to be the lattice norm complexification (see
\cite[Section~II.11]{Schaefer1974}) of the gauge norm $\|\phdot\|_u$ on
$(E_{\mathbb R})_u$. When endowed with the gauge norm, $E_u$ embeds
continuously into $E$.  What is important for our purposes is that
$E_u$ with the gauge norm is a Banach lattice and we have an isometric
lattice isomorphism
\begin{displaymath}
  E_u\cong C(K)
\end{displaymath}
for some compact Hausdorff space $K$; see \cite[Corollary to Prop~II.7.2
and Theorem~II.7.4]{Schaefer1974}).

We call $u \in E_+$ a \emph{quasi-interior point} of $E_+$ and write $u
\gg 0$ if $E_u$ is dense in $E$. If $u$ is a quasi-interior point of
$E_+$, then we say that $f \in E$ is \emph{strongly positive with
  respect to $u$} and write $f \gg_u 0$ if there is a $c > 0$ such that
$f \ge cu$. We say $f$ is \emph{strongly negative with respect to $u$}
and write $f \ll_u 0$ if $-f \gg_u0$. Actually, those definitions make
sense for arbitrary positive vectors $u$, but we shall only need them
in case that $u$ is a quasi-interior point of $E_+$.

An operator $T \in \calL(E,F)$ is called \emph{positive} if
$TE_+\subseteq F_+$; we say that $T$ is \emph{strongly positive} and 
denote this by $T \gg 0$ if $Tf \gg 0$ for all $f >0$. 
Given a quasi-interior point $u
\in E_+$, we say that $T$ is \emph{strongly positive with respect to
  $u$} and write $T \gg_u 0$ if $Tf \gg_u 0$ whenever $f > 0$.  We call
$T$ \emph{strongly negative with respect to $u$} and write $T \ll_u 0$
if $-T \gg_u 0$. A positive operator $T \in \calL(E)$ is called
\emph{irreducible} if $\{0\}$ and $E$ are the only $T$-invariant closed
ideals in $E$.

The dual space of $E$ is denoted by $E'$ and is again a Banach lattice.
A linear functional $\varphi \in E'$ is called \emph{strictly positive}
if $\langle \varphi, f \rangle > 0$ for every $f > 0$, i.e.\ $\varphi$ is
strictly positive if and only if it is strongly positive as an operator from
$E$ to $\bbR$ (or $\bbC$). Note that
$\varphi$ is automatically strictly positive if it is a quasi-interior
point of $(E')_+$, but the converse is not true. We should point out that 
this causes the following ambiguity in our notation: if we write $\varphi \gg 0$ for
a functional $\varphi \in E'$, then this could either mean that $\varphi$
is a quasi-interior point of $E'_+$ or that $\varphi$ is strongly positive
as an operator from $E$ to the scalar field, i.e.\ that $\varphi$ is strictly 
positive. For this reason, we never use the notation $\gg$ for functionals.

\paragraph{Linear operators, resolvent and spectrum}
The domain of an operator $A$ on a Banach space $E$ will always be
denoted by $D(A)$, and if not stated otherwise, $D(A)$ will be assumed
to be endowed with the graph norm. If $A$ is densely defined, then its 
adjoint is well defined and we denote it by $A'$. Let $E$ and $F$ be
complex Banach lattices, i.e.~let them be complexifications of real
Banach lattices $E_{\mathbb R}$ and $F_{\mathbb R}$. We call an operator
$A\colon D(A)\subseteq E\to F$ \emph{real} if $D(A) = D(A) \cap
E_{\mathbb R} + i D(A) \cap E_{\mathbb R}$ and if $A (D(A) \cap
E_{\mathbb R})\subseteq F_{\mathbb R}$. Positive and, in particular,
strongly positive operators are automatically real.

Let $A$ be a closed linear operator on a complex Banach space $E$; we
denote its \emph{spectrum} by $\sigma(A)$, its \emph{resolvent set} by
$\rho(A) := \bbC \setminus \sigma(A)$, and for each $\lambda \in
\rho(A)$ the operator $R(\lambda,A) := (\lambda\id-A)^{-1}$ denotes the
\emph{resolvent} of $A$ at $\lambda$. The \emph{spectral bound} of $A$
is given by
\begin{displaymath}
  \spb(A):=\sup\bigl\{\repart \lambda\colon \lambda \in \sigma(A)\bigr\}
  \in [-\infty, \infty].    
\end{displaymath}  
If $\spb(A) \in \bbR$, the set
\begin{displaymath}
  \sigma_{\per}(A) := \sigma(A) \cap (\spb(A) + i\bbR)
\end{displaymath}
is called the \emph{peripheral spectrum} of $A$. We call $\spb(A)$ a
\emph{dominant spectral value} of $A$ if $\sigma_{\per}(A) =
\{\spb(A)\}$. In particular, this includes the assertion that $\spb(A)
\in \sigma(A)$.

A spectral value $\mu \in \sigma(A)$ is called a \emph{pole of the
  resolvent} of $A$ if the analytic mapping $\rho(A)\to\calL(E)$,
$\lambda\mapsto R(\lambda,A)$ has a pole at $\mu$.  We will make
extensive use of the Laurent series expansion of $R(\phdot,A)$ about its
poles and of the spectral projection $P$ associated with a pole of
$R(\phdot,A)$. Details on this can
be found in \cite[Section~2]{DanersI}, \cite[Section-III.6.5]{Kato1976},
\cite[Section~VIII.8]{Yosida1995}, \cite[Section~IV.1]{Engel2000} or
\cite{campbell:13:lac}.

\paragraph{Semigroups}
Suppose the operator $A$ on a Banach space $E$ generates a
$C_0$-semigroup, which will be denoted by $(e^{tA})_{t \ge 0}$. This
semigroup is called \emph{eventually differentiable} if there is a $t_0
> 0$ such that $e^{t_0A}E \subseteq D(A)$, \emph{eventually norm
  continuous} if there is $t_0 \ge 0$ such that the mapping
$[t_0,\infty) \to \calL(E)$, $t\mapsto e^{tA}$ is continuous with
respect to the operator norm on $\calL(E)$, and \emph{uniformly
  exponentially stable} if $\| e^{tA} \|_{\calL(E)} \to 0$ as $t \to
\infty$. A $C_0$-semigroup $(e^{tA})_{t \ge 0}$ on a complex Banach 
lattice $E$ is called \emph{real} if each operator $e^{tA}$ is real. It is easy 
to check that a $C_0$-semigroup on $E$ is real if and only if its generator
is real.

\section{Strongly positive projections}
\label{section:projections-strong}
In this section, we consider eigenvalues of linear operators on complex
Banach lattices and characterise when the corresponding spectral
projection is strongly positive. Our first result is the following
analogue of \cite[Proposition~3.1]{DanersI} for arbitrary Banach
lattices.

\begin{proposition}
  \label{prop:projections-strong}
  Let $A$ be a closed, densely defined and real operator on a complex
  Banach lattice $E$, let $\lambda_0 \in \bbR$ be an eigenvalue of $A$
  and a pole of $R(\phdot,A)$, and let $P$ be the corresponding spectral
  projection. Then the following assertions are equivalent:
  \begin{enumerate}[\upshape (i)]
  \item $P$ is positive and irreducible.
  \item $P \gg 0$.
  \item The eigenvalue $\lambda_0$ of $A$ is geometrically
    simple. Moreover, $\ker(\lambda_0\id - A)$ contains a quasi-interior
    point of $E_+$ and $\ker(\lambda_0\id - A')$ contains a strictly
    positive vector.
  \item The eigenvalue $\lambda_0$ of $A$ is algebraically simple,
    $\ker(\lambda_0\id-A)$ contains a quasi-interior point of $E_+$ and
    $\im(\lambda_0\id-A) \cap E_+ = \{0\}$.
  \end{enumerate}
  If assertions {\upshape (i)--(iv)} are fulfilled, then $\lambda_0$ is
  a simple pole of the resolvents $R(\phdot,A)$ and $R(\phdot,A')$ and
  $\lambda_0$ is the only eigenvalue of $A$ having a positive
  eigenvector.
\end{proposition}

The proof requires some properties of positive projections which  
are given in the next lemma and which are based on 
standard techniques from the Perron--Frobenius theory of
positive operators; compare \cite[Sections~V.4 and~V.5]{Schaefer1974}.

\begin{lemma}
  \label{lem:positive-projections}
  Let $E$ be a complex Banach lattice and let $P \in \calL(E)$ be a
  positive, non-zero and irreducible projection. Then every non-zero
  element of $E_+\cap\im P$ is a quasi-interior point of $E_+$ and $\im
  P'$ contains a strictly positive functional. Moreover, $\dim(\im
  P)=\dim(\im P')=1$.
\end{lemma}
\begin{proof}
  Since $P$ is positive and non-zero, $\im P$ contains a vector $u >
  0$. As $\overline{E_u}$ is $P$-invariant and $P$ is irreducible we
  conclude that $E_u$ is dense in $E$ and hence $u\gg 0$. In fact, this
  argument shows that every non-zero element of $\im(P)\cap E_+$
  is a quasi-interior point of $E_+$. Since the adjoint projection $P'$
  is also positive and non-zero, $\im P'$ contains a functional $\varphi
  > 0$. As $P$ is irreducible, we easily conclude that $\varphi$ is strictly
  positive.
 
  Note that $\dim(\im P)=\dim(\im P')$ (see~\cite[Section~III.6.6]{Kato1976}), 
  so it remains to show that
  $\dim(\im P)=1$. To this end, let us show first that $|v| \in \im P$ 
  whenever $v\in\im P$. If
  $v\in\im P$ and if $\varphi\in\im P'$ is strictly positive, then
  $P|v|\ge|Pv|=|v|$ and hence
  \begin{displaymath}
    0 \le \langle \varphi, P|v| - |v| \rangle
    = \langle P'\varphi,|v|\rangle - \langle\varphi,|v|\rangle
    = 0 \text{.}
  \end{displaymath}
  As $\varphi$ is strictly positive we conclude that $P|v| = |v|$, 
  so indeed $|v| \in \im P$.
  
  By definition, the complex Banach lattice $E$ is the
  complexification of a real Banach lattice $E_\bbR$. If we define
  $F_\bbR := E_\bbR \cap \im P$, then $\im P = F_\bbR + i F_\bbR$ since
  $PE_\bbR \subseteq E_\bbR$. Hence it is sufficient to show that
  $F_\bbR$ is one-dimensional over $\bbR$. We have shown that $|v| \in
  F_\bbR$ for each $v \in F_\bbR$, and so $F_\bbR$ is a sublattice of
  $E_\bbR$. Clearly $F_\bbR$ is a normed vector lattice with respect to
  the norm induced by $E_\bbR$ and therefore it is Archimedean (see
  \cite[Proposition~II.5.2(ii)]{Schaefer1974}). Hence, to show that
  $F_\bbR$ is one-dimensional we only need to show that $F_\bbR$ is
  totally ordered, see \cite[Proposition~II.3.4]{Schaefer1974}.  To do
  so, let $v_1,v_2 \in F_\bbR$ and set $v := v_1-v_2$. Then the positive
  part $v^+$ lies in $\im P$ and by what we have shown above either
  $v^+=0$ or $v^+\gg 0$. Hence, either $v_1-v_2\leq 0$ or $v_1-v_2\gg
  0$, showing that $F_\bbR$ is totally ordered and thus one-dimensional
  over $\bbR$.
\end{proof}
Regarding the assumptions of Lemma~\ref{lem:positive-projections}, we
point out that a positive irreducible projection on a Banach
lattice $E$ is automatically non-zero whenever $\dim E \ge 2$.
\begin{proof}[Proof of Proposition~\ref{prop:projections-strong}]
  We may assume without loss of generality that $\lambda_0 = 0$, since
  otherwise we may replace $A$ with $A - \lambda_0\id$. We prove (i)
  $\Rightarrow$ (iii) $\Rightarrow$ (iv) $\Rightarrow$ (ii)
  $\Rightarrow$ (i).
 
  ``(i) $\Rightarrow$ (iii)'' If (i) holds, then
  Lemma~\ref{lem:positive-projections} implies (iii). The lemma in
  particular asserts that $\dim(\im P)=1$, that is, $\lambda_0 = 0$ is
  algebraically and hence geometrically simple.
 
  ``(iii) $\Rightarrow$ (iv)'' By (iii), $0$ is a geometrically simple
  eigenvalue. To show that it is algebraically simple we have to prove
  that $\ker A^2=\ker A$.  Let $u \in \ker A^2$. Then $Au \in \ker A$
  and by (iii) there exists $\alpha \in \bbC \setminus \{0\}$ such that
  $\alpha Au\geq 0$. By assumption there is a strictly positive
  functional $\varphi \in \ker A'$. Thus
  \begin{math}
    \langle \varphi, \alpha Au \rangle = \langle A'\varphi, \alpha u
    \rangle = 0 \text{.}
  \end{math}
  As $\varphi$ is strictly positive, we conclude that $\alpha Au = 0$,
  that is, $u \in \ker A$ as claimed.  Finally, let $v = Au \in E_+ \cap
  \im A$. Then $\langle \varphi, v \rangle = \langle A'\varphi, u
  \rangle = 0$ and thus $v = 0$.
 
  ``(iv) $\Rightarrow$ (ii)'' Since $0$ is algebraically simple we have
  $\ker P = \im A$, $\im P = \ker A$ and $E = \im P \oplus \ker
  P$. Hence, we can decompose each $f \in E_+ \setminus \{0\}$ uniquely
  in the form $f = Pf + g$, where $Pf \in \ker A$ and $g \in \im
  A$. From (iv) we have that $Pf = \alpha u$ for a quasi-interior point
  $u \in E_+$ and for some scalar $\alpha \in \bbC$. Since $A$ is real,
  so is $P$, and hence we have $\alpha \in \bbR$.  Assume for a
  contradiction that $\alpha \le 0$. Then $0<f\leq f - \alpha u =
  g\in\im A$ which contradicts (iv). Hence we must have $\alpha > 0$ and
  thus $Pf = \alpha u \gg 0$.

  The implication ``(ii) $\Rightarrow$ (i)'' is obvious.
 
  Now assume that the equivalent assertions (i)--(iv) hold. By
  Lemma~\ref{lem:positive-projections} we have $\dim \im P = \dim \im P'
  = 1$. Hence $\lambda_0$ is an algebraically simple eigenvalue of $A$
  and $A'$ and thus a simple pole of $R(\phdot,A)$ and
  $R(\phdot,A')$. Finally, let $\lambda \in \bbC \setminus \{0\}$ be an
  eigenvalue of $A$ and $u$ a corresponding eigenvector. Then $0 \not= u =
  \lambda^{-1}Au\in\im A$. As $E_+\cap\im A=\{0\}$ by (iv), $u$ cannot
  be positive.
\end{proof}
The reader may find some related arguments in the proof of
\cite[Remark~B-III.2.15(a)]{Arendt1986}.
As pointed out above, Proposition~\ref{prop:projections-strong} is
an analogous result to \cite[Proposition~3.1]{DanersI} where the situation 
on $C(K)$-spaces was considered. 
However, $u\gg 0$ in $C(K)$ means that $u$ is an \emph{interior} point
of the positive cone, whereas in a general Banach lattice the interior
of the positive cone is empty. This is the main obstacle when seeking to
generalise the results from \cite{DanersI}. For this reason we will not
focus on the relation $\gg$, but on the stronger property of being
strongly positive \emph{with respect to a given quasi-interior point}
(see Section~\ref{section:notation} for details). The following
corollary translates Proposition~\ref{prop:projections-strong} into this
setting.
\begin{corollary}
  \label{cor:projections-strong}
  Let $A$ be a closed, densely defined and real operator on a complex
  Banach lattice $E$. Let $\lambda_0 \in \bbR$ be an eigenvalue of $A$
  and a pole of the resolvent and denote by $P$ the corresponding
  spectral projection. If $u \in E_+$ is a quasi-interior point, then
  the following assertions are equivalent:
  \begin{enumerate}[\upshape (i)]
  \item $P \gg_u 0$.
  \item The eigenvalue $\lambda_0$ of $A$ is geometrically simple.
    Moreover, $\ker(\lambda_0\id-A)$ contains a vector $x \gg_u 0$ and
    $\ker(\lambda_0\id-A')$ contains a strictly positive vector.
  \item The eigenvalue $\lambda_0$ of $A$ is algebraically simple,
    $\ker(\lambda_0\id-A)$ contains a vector $x \gg_u 0$ and
    $\im(\lambda_0\id-A) \cap E_+ = \{0\}$.
  \end{enumerate}
  If assertions {\upshape (i)--(iii)} are fulfilled, then $\lambda_0$ is
  a simple pole of the resolvents $R(\phdot,A)$ and $R(\phdot,A')$ and 
  $\lambda_0$ is the only eigenvalue of $A$ having a positive
  eigenvector.
\end{corollary}
\begin{proof}
  We assume throughout the proof that $\lambda_0 = 0$.
 
  ``(i) $\Rightarrow$ (ii)'' Assertion (i) clearly implies that $P \gg
  0$. Hence by Proposition~\ref{prop:projections-strong} we only have to
  show that $\ker A$ contains a vector $v \gg_u 0$. We already know that
  $\ker A$ contains a quasi-interior point $v \gg 0$. As $v \in \im P$
  we indeed have $v = Pv \gg_u 0$.
 
  ``(ii) $\Rightarrow$ (iii)'' By (ii) we already know that $\ker A$
  contains a vector $v \gg_u 0$. The remaining assertions follow from
  Proposition~\ref{prop:projections-strong}.
 
  ``(iii) $\Rightarrow$ (i)'' Let $v\in\ker A$ with $v\gg_u 0$. If
  $f>0$, then $Pf \gg 0$ by
  Proposition~\ref{prop:projections-strong}. As $0$ is algebraically
  simple we have $\im P = \ker A$ and $\dim\im P=1$.  Hence, $Pf =
  \alpha v$ for some $\alpha \in \bbC$.  As $Pf \gg 0$ we see that
  $\alpha > 0$. Thus $Pf \gg_u 0$.
  
  Suppose now that (i)--(iii) are fulfilled. Due to (i) we clearly have
  $Pf \gg 0$ for every $f>0$ and hence the remaining assertions follow from
  Proposition~\ref{prop:projections-strong}.
\end{proof}

\section{Eventually strongly positive resolvents}
\label{section:resolvents-strong}
To prepare for our analysis of eventually positive semigroups, we first
consider what we shall call eventually positive resolvents.  Here we
will generalise certain results on $C(K)$-spaces from
\cite[Section~4]{DanersI} to the technically more demanding case of
general Banach lattices. As pointed out before
Corollary~\ref{cor:projections-strong} it seems appropriate in this
setting not to consider merely strong positivity, but strong positivity
with respect to a fixed quasi-interior point.

\begin{definition}
  \label{def:resolvents-strong}
  Let $A$ be a closed, real operator on a complex Banach lattice $E$,
  let $u \in E_+$ be a quasi-interior point and let $\lambda_0$ be
  either $-\infty$ or a spectral value of $A$ in $\bbR$.
  \begin{enumerate}[(a)]
  \item The resolvent $R(\phdot,A)$ is called \emph{individually
      eventually strongly positive with respect to $u$ at $\lambda_0$}
    if there exists $\lambda_2 > \lambda_0$ with the following
    properties: $(\lambda_0,\lambda_2] \subseteq \rho(A)$ and for each $f
    \in E_+\setminus \{0\}$ there is a $\lambda_1 \in
    (\lambda_0,\lambda_2]$ such that $R(\lambda,A)f \gg_u 0$ for all
    $\lambda \in (\lambda_0,\lambda_1]$.
  \item The resolvent $R(\phdot,A)$ is called \emph{uniformly eventually
      strongly positive with respect to $u$ at $\lambda_0$} if there
    exists $\lambda_1 > \lambda_0$ with the following properties:
    $(\lambda_0,\lambda_1] \subseteq \rho(A)$ and $R(\lambda,A) \gg_u 0$
    for all $\lambda \in (\lambda_0,\lambda_1]$.
  \end{enumerate}
\end{definition}
While eventual positivity focuses on what happens to the resolvent in a
right neighbourhood of a spectral value, we might also ask what happens
in a left neighbourhood. As we will see, \emph{eventual negativity} is
the appropriate notion to describe this behaviour in our setting.

\begin{definition}
  \label{def:resolvents-strong-negative}
  Let $A$ be a closed, real operator on a complex Banach lattice $E$,
  let $u \in E_+$ be a quasi-interior point and let $\lambda_0$ be
  either $\infty$ or a spectral value of $A$ in $\bbR$.
  \begin{enumerate}[(a)]
  \item The resolvent $R(\phdot,A)$ is called \emph{individually
      eventually strongly negative with respect to $u$ at $\lambda_0$}
    if there exists $\lambda_2 < \lambda_0$ with the following
    properties: $[\lambda_2,\lambda_0) \subseteq \rho(A)$ and for each $f
    \in E_+\setminus \{0\}$ there is a $\lambda_1 \in
    [\lambda_2,\lambda_0)$ such that $R(\lambda,A)f \ll_u 0$ for all
    $\lambda \in [\lambda_2,\lambda_0)$.
  \item The resolvent $R(\phdot,A)$ is called \emph{uniformly eventually
      strongly negative with respect to $u$ at $\lambda_0$} if there
    exists $\lambda_1 < \lambda_0$ with the following properties:
    $[\lambda_1,\lambda_0) \subseteq \rho(A)$ and $R(\lambda,A) \ll_u 0$
    for all $\lambda \in [\lambda_1,\lambda_0)$.
  \end{enumerate}
\end{definition}

Concerning eventual strong positivity of resolvents (with respect to a
quasi-interior point) we can make similar observations on arbitrary
Banach lattices as were made for $C(K)$-spaces in \cite[Propositions~4.2
and~4.3]{DanersI}. However, we do not pursue this in detail
here. Instead, we concentrate on proving a characterisation of
individually eventually strongly positive resolvents. To state this
characterisation, the following notion concerning powers of a given
operator will be useful.

\begin{definition}
  \label{def:powers-strong}
  Let $T$ be a bounded linear operator on a complex Banach lattice $E$
  and let $u \in E_+$ be a quasi-interior point.
  \begin{enumerate}[(a)]
  \item The operator $T$ is called \emph{individually eventually
      strongly positive with respect to $u$} if for every $f \in E_+
    \setminus \{0\}$ there is an $n_0 \in \bbN$ such that $T^nf \gg_u 0$
    for all $n \ge n_0$.
  \item The operator $T$ is called \emph{uniformly eventually strongly
      positive with respect to $u$} if there is an $n_0 \in \bbN$ such
    that $T^n \gg_u 0$ for all $n \ge n_0$.
  \end{enumerate}
\end{definition}

We can now formulate the following theorem, which was previously known
only on $C(K)$-spaces, cf.~\cite[Theorem~4.5]{DanersI}. Beside its wider
applicability, the main difference in the case of general Banach
lattices is that the cone might not contain interior points; but if
instead we take a quasi-interior point $u$, then we cannot control other
vectors by $u$ unless they are contained in the principal ideal $E_u$.
This makes it necessary to impose an additional domination hypothesis on 
$D(A)$, which in turn requires more technical proofs; nevertheless, in many
practical examples, this condition seems to be satisfied, suggesting it
is in a sense quite natural. See also the discussion below.

\begin{theorem}
  \label{thm:resolvents-strong}
  Let $A$ be a closed, densely defined and real operator on a complex
  Banach lattice $E$.  Suppose that $\lambda_0 \in \bbR$ is an
  eigenvalue of $A$ and a pole of the resolvent. Denote by $P$ the
  corresponding spectral projection. Moreover, let $0 \le u \in E$ and
  assume that $D(A) \subseteq E_u$. Then $u$ is a quasi-interior point
  of $E_+$ and the following assertions are equivalent:
  \begin{enumerate}[\upshape (i)]
  \item $P \gg_u 0$.
  \item The resolvent $R(\phdot,A)$ is individually eventually strongly
    positive with respect to $u$ at $\lambda_0$.
  \item The resolvent $R(\phdot,A)$ is individually eventually strongly
    negative with respect to $u$ at $\lambda_0$.
  \end{enumerate}
  If $\lambda_0 = \spb(A)$, then {\upshape (i)}--{\upshape (iii)} are
  also equivalent to the following assertions.
  \begin{enumerate}[\upshape (i)]\setcounter{enumi}{3}
  \item There exists $\lambda > \spb(A)$ such that the operator
    $R(\lambda,A)$ is individually eventually strongly positive with respect to
    $u$.
  \item For every $\lambda > \spb(A)$ the operator $R(\lambda,A)$ is
    individually eventually strongly positive with respect to $u$.
  \end{enumerate}
\end{theorem}

Before we prove the above theorem, a few remarks on the condition $D(A)
\subseteq E_u$ are in order.  First, if we endow $E_u$ with the
(complexification of the) gauge norm, then the embedding $D(A)
\hookrightarrow E_u$ is automatically continuous due to the closed graph
theorem, a fact of which we will make repeated use. Second,
it is natural to ask whether the condition $D(A) \subseteq E_u$ in the
above theorem can be omitted, but
Example~\ref{example:domination-property} below shows that it is
required. Finally, one might wonder how to check this condition in
applications. In a typical situation, the Banach lattice $E$ is an
$L^p(\Omega)$-space on a finite measure space $\Omega$, $u$ is the
constant function $1$ and the principal ideal $E_u$ is thus given by
$L^\infty(\Omega)$. If $\Omega$ is a bounded open set in $\bbR^n$ and $A$ is a
differential operator which is defined on some Sobolev space, then the
condition $D(A) \subseteq E_u$ is fulfilled if an appropriate Sobolev
embedding theorem holds.  Some concrete examples of this type can be
found in Section~\ref{section:applications-strong}, but see also the
following example and remark.

\begin{example}
  \label{ex:anti-maximum-principle}
  The above theorem contains what is often referred to as an
  \emph{anti-maximum principle}. Let $A$ be a closed densely defined
  real operator on the Banach Lattice $E$. Suppose that
  $\lambda_0\in\mathbb R$ is a pole of the resolvent $R(\phdot,A)$ with
  spectral projection $P$. We consider the equation
  \begin{equation}
    \label{eq:anti-maximum-principle}
    \lambda f-Af=g\qquad\text{in $E$}
  \end{equation}
  with $\lambda<\lambda_0$ close to $\lambda_0$ and $g>0$. If we assume
  that there exists $u\gg 0$ such that $D(A)\subseteq E_u$ and $P\gg_u
  0$, then Theorem~\ref{thm:resolvents-strong}(iii) implies that for
  every $g>0$ there exists $\lambda_g<\lambda_0$ such that the solution
  $f$ of \eqref{eq:anti-maximum-principle} satisfies $f\ll_u0$ whenever
  $\lambda\in(\lambda_g,\lambda_0)$. This is known as a (non-uniform)
  anti-maximum principle and has been the focus of many papers such as
  \cite{MR0550042,MR1761420,MR1896075,MR1396904}, looking at standard
  second order elliptic equations, but also higher order elliptic
  equations of order $2m$. The key assumptions in our language are that
  $P\gg_u0$ and that $D(A)\subseteq E_u$. The first one is most
  conveniently obtained by checking
  Proposition~\ref{prop:projections-strong}(iii). In many applications
  the dual problem has the same structure as the original problem, and
  therefore guarantees the existence of a positive eigenfunction for
  both. The condition that $D(A)\subseteq E_u$ follows from elliptic
  regularity theory as well as boundary maximum principles. If these
  regularity conditions are violated, an anti-maximum principle may fail
  as shown in \cite{MR1429095}.
\end{example}

\begin{remark}
  The known anti-maximum principles also show that \emph{uniform} strong
  eventual positivity is not in general equivalent to \emph{uniform}
  strong eventual negativity of the resolvent. As an example, for second
  order elliptic boundary value problems, the strong maximum principle
  implies $R(\lambda,A) \gg 0$ for all $\lambda>\spb(A)$. 
  However, as
  shown in \cite{MR1896075}, the anti-maximum principle is not
  necessarily uniform. At an abstract level, it does not seem to be
  obvious what guarantees uniform strong eventual negativity (or
  positivity); in \cite{MR1896075} it is a certain kernel estimate that
  does this.
\end{remark}
 
For the proof of Theorem~\ref{thm:resolvents-strong} we need a few
lemmata which are generalisations and extensions of similar auxiliary
results in \cite[Section~4]{DanersI}. In particular, in the next lemma
we obtain convergence in a stronger norm than in
\cite[Lemma~4.7(ii)]{DanersI}.
\begin{lemma}
  \label{lem:behaviour-resolvent}
  Let $A$ be a closed linear operator on a complex Banach space $E$.
  Suppose that $0$ is an eigenvalue of $A$ and a simple pole of
  $R(\phdot,A)$. Let $P$ be the corresponding spectral projection.
  \begin{enumerate}[\upshape (i)]
  \item We have $\lambda R(\lambda,A)\to P$ in $\calL(E,D(A))$ as
    $\lambda \to 0$.
  \item If in addition $0=\spb(A)$, then for every $\lambda>0$ we have
    $[\lambda R(\lambda,A)]^n \to P$ in $\calL(E,D(A))$ as $n\to\infty$.
  \end{enumerate}
\end{lemma}
\begin{proof}
  (i) As $0$ is a simple pole of $R(\lambda,A)$ and $P$ is the
  corresponding residue, $\lambda R(\lambda,A)\to P$ in $\calL(E)$ as
  $\lambda \to 0$ and $\im P=\ker A$. Therefore
  \begin{displaymath}
    A\lambda R(\lambda,A)
    =\lambda\bigl(\lambda R(\lambda,A)-\id\bigr)
    \to 0=AP
  \end{displaymath}
  in $\calL(E)$ as $\lambda\downarrow 0$ and so the required convergence holds
  in $\calL(E,D(A))$.
 
  (ii) By \cite[Lemma~4.7(ii)]{DanersI} and its proof we have that the
  convergence holds in $\calL(E)$, and that $\lambda R(\lambda,A)P=P$. Since
  $R(\lambda,A)\in \calL(E,D(A))$ due to the closed graph theorem, it 
  follows that
  \begin{displaymath}
    [\lambda R(\lambda,A)]^n 
    =\lambda R(\lambda,A) [\lambda R(\lambda,A)]^{n-1}
    \to\lambda R(\lambda,A) P = P
  \end{displaymath}
  as $n\to\infty$ in $\calL(E,D(A))$.
\end{proof}

\begin{lemma}
  \label{lem:operator-family}
  Let $E$ be a complex Banach lattice and let $u \in E_+$ be a
  quasi-interior point. Let $(J,\preceq)$ be a non-empty, totally
  ordered set and let $\mathcal{T} = (T_j)_{j \in J}$ be a family in
  $\calL(E)$ whose fixed space is denoted by
  \begin{displaymath}
    F := \{v \in E\colon T_j v = v \text{ for all $j \in J$}\}
  \end{displaymath}
  Assume that for every $f \in E_+ \setminus \{0\}$ there exists $j_f
  \in J$ such that $T_jf \gg_u 0$ for all $j \succeq j_f$.
  \begin{enumerate}[\upshape (i)]
  \item Suppose that for every $j_0\in J$ the family $(T_j|_{E_u})_{j
      \preceq j_0}$ is bounded in $\calL(E_u,E)$ and that $F$ contains
    an element $v_0 > 0$. Then the entire family $(T_j|_{E_u})_{j \in
      J}$ is bounded in $\calL(E_u,E)$.
  \item Let $P > 0$ be a projection in $\calL(E)$ with $\im P \subseteq
    F$ and suppose that each operator $T_j$, $j \in J$, leaves $\ker P$
    invariant. Then $P \gg_u 0$.
  \end{enumerate}
\end{lemma}
\begin{proof}
  (i) By the uniform boundedness principle we only have to show that
  $(T_jf)_{j \in J}$ is bounded in $E$ for every $0 < f\in E_u$. By
  assumption there is a vector $0 < v_0 \in F$ and $j_0\in J$ such that
  $v_0 = T_{j_{0}}v_0 \gg_u 0$. For $0 < f \in E_u$ we can thus find a
  constant $c \ge 0$ such that $cv_0 \pm f \ge 0$. Hence we have
  $T_j(cv_0 \pm f) \ge 0$ and thus $|T_jf| \le cv_0$ for all
  sufficiently large $j$. This yields the assertion.
 
  (ii) If $f > 0$ then $Pf \in F$ and $Pf \ge 0$. In the case that
  $Pf\neq 0$ we have $Pf = T_jPf \gg_u 0$ for some $j \in J$. To show
  that $Pf \neq 0$ for every $f > 0$ fix $f > 0$. As $P\neq 0$ and $E_u$
  is dense in $E$ there is an element $0 < g \in E_u$ such that $Pg >
  0$. By assumption $T_jf \gg_u 0$ for some $j\in J$, and thus $T_jf -
  cg \ge 0$ for some $c > 0$. Hence, $PT_jf \ge cPg > 0$ and in
  particular $PT_jf \neq 0$. Since $T_j$ leaves $\ker P$ invariant, this
  implies that $Pf \not= 0$.
\end{proof}

We are now able to prove Theorem~\ref{thm:resolvents-strong}.

\begin{proof}[Proof of Theorem~\ref{thm:resolvents-strong}]
  We may assume throughout the proof that $\lambda_0 = 0$. First,
  observe that the domination condition $D(A)\subseteq E_u$ implies that
  $u$ is a quasi-interior point of $E_+$ since $D(A)$ is dense in
  $E$. We shall prove (i) $\Leftrightarrow$ (ii), (i) $\Leftrightarrow$
  (iii) and (i) $\Rightarrow$ (v) $\Rightarrow$ (iv) $\Rightarrow$ (i).
 
  ``(i) $\Rightarrow$ (ii)'' If (i) holds, then
  Proposition~\ref{prop:projections-strong} implies that $0$ is a simple
  pole of $R(\phdot,A)$. Let $f>0$.
  Lemma~\ref{lem:behaviour-resolvent}(i) now yields that $\lambda
  R(\lambda,A)f \to Pf \gg_u 0$ in $D(A)$ as $\lambda \downarrow 0$. By
  the closed graph theorem, $D(A)\hookrightarrow E_u$ if $E_u$ is
  endowed with the gauge norm with respect to $u$. Hence, $\lambda
  R(\lambda,A)f \to Pf \gg_u 0$ in $E_u$. Since $\lambda
  R(\lambda,A)f\in E_{\bbR}$ for every $\lambda\in(0,\infty)\cap\rho(A)$
  and $E_u=C(K)$ for some compact Hausdorff space $K$, this implies that
  $\lambda R(\lambda,A)f \gg_u 0$ for sufficiently small $\lambda > 0$.
  
  ``(ii) $\Rightarrow$ (i)'' First we show that the eigenvalue $0$ of
  $A$ admits a positive eigenvector. Let $m \ge 1$ be the order of $0$
  as a pole of $R(\phdot,A)$ and let $U\in \calL(E)$ be the coefficient
  of $\lambda^{-m}$ in the Laurent expansion of $R(\phdot,A)$ about
  $0$. Then $U\neq 0$ and $\im U$ consists of eigenvectors of $A$ (see
  for instance \cite[Remark~2.1]{DanersI}).  Moreover, $\lambda^m
  R(\lambda,A) \to U$ in $\calL(E)$. As the resolvent is individually
  eventually positive, $U\geq 0$ and so $A$ has an eigenvector $v > 0$
  corresponding to the eigenvalue $0$.
  
  We now apply Lemma~\ref{lem:operator-family} to the operator family
  $(\lambda R(\lambda,A))_{\lambda \in (0,\varepsilon]}$, where 
  $\varepsilon > 0$ is sufficiently small to ensure that $(0,\varepsilon] \in 
  \rho(A)$ and where the order on $(0,\varepsilon]$ 
  is given by the relation $\preceq := \ge$. Note that the fixed
  space of this operator family coincides with $\ker A$.  Therefore, all
  assumptions of part (i) of the Lemma are fulfilled, and we conclude
  that the operator family $(\lambda R(\lambda,A)|_{E_u})_{\lambda \in
    (0,\varepsilon]}$ is bounded in $\calL(E_u,E)$.
 
  If we fix some $\mu \in \rho(A)$, then by the resolvent identity
  \begin{displaymath}
    \lambda R(\lambda,A)
    = (\mu - \lambda)\lambda R(\lambda,A) R(\mu,A)+\lambda R(\mu,A)
  \end{displaymath}
  for all $\lambda \in (0,\varepsilon]$. We have $R(\mu,A)\in\calL(E,D(A))$. As
  $D(A)\hookrightarrow E_u$, we conclude that
  $R(\mu,A)\in\calL(E,E_u)$. Hence we have
  \begin{displaymath}
    \|\lambda R(\lambda,A)\|_{\calL(E)} \le |\mu - \lambda| \, 
    \|\lambda R(\lambda,A)|_{E_u}\|_{\calL(E_u,E)} \, 
    \|R(\mu,A)\|_{\calL(E,E_u)} + |\lambda| \, \|R(\mu,A)\|_{\calL(E)}
  \end{displaymath}
  for every $\lambda \in (0,\varepsilon]$. The operator family $(\lambda
  R(\lambda,A))_{\lambda \in (0,\varepsilon]}$ is therefore bounded in $\calL(E)$,
  showing that $0$ is a simple pole of $R(\phdot,A)$.
 
  Hence, we have $P = U > 0$. Moreover, $\im P=\ker A$ and therefore the
  fixed space of $(\lambda R(\lambda,A))_{\lambda \in (0,\varepsilon]}$ is $\im
  P$. Applying Lemma~\ref{lem:operator-family} we conclude that $P \gg_u
  0$.
  
  ``(i) $\Leftrightarrow$ (iii)'' Note that $0$ is also an eigenvalue of
  $-A$ and that the corresponding spectral projection is also $P$. Thus,
  (i) holds if and only if $R(\phdot,-A)$ is individually
  eventually strongly positive with respect to $u$ at $0$. This however
  is true if and only if $R(\phdot,A)$ is individually eventually
  strongly negative with respect to $u$ at $0$.
  
  From now on we assume that $\lambda_0=\spb(A)=0$.
  
  ``(i) $\Rightarrow$ (v)'' We argue similarly as in the implication
  ``(i) $\Rightarrow$ (ii)'', but use part (ii) of
  Lemma~\ref{lem:behaviour-resolvent} instead of part (i) to conclude
  that for every $f > 0$ we have $[\lambda R(\lambda,A)]^n f\gg_u
  0$ for all $n$ sufficiently large.
 
  ``(v) $\Rightarrow$ (iv)'' This implication is obvious.
  
  ``(iv) $\Rightarrow$ (i)'' We proceed similarly as in the proof of
  ``(ii) $\Rightarrow$ (i)'', so we only provide an outline. Let
  $\lambda>0$ such that $R(\lambda,A)$ is individually strongly positive
  with respect to $u$. We apply \cite[Lemma~4.8]{DanersI} to $T :=
  \lambda R(\lambda,A)$ to conclude that $\lambda R(\lambda, A)$ admits
  an eigenvector $v > 0$ for the eigenvalue $1$. Then we apply
  Lemma~\ref{lem:operator-family}(i) to the operator family $([\lambda
  R(\lambda,A)]^n)_{n \in \bbN_0}$ and conclude that its restriction to
  $E_u$ is bounded in $\calL(E_u,E)$.  Since $\lambda
  R(\lambda,A)\in\calL(E,E_u)$, the family is also bounded in $\calL(E)$
  and hence $0$ is a simple pole of $R(\phdot,A)$. It follows from
  Lemma~\ref{lem:behaviour-resolvent}(ii) that $P$ is positive and since
  $0$ is a spectral value of $A$, $P$ is non-zero. As above,
  we can now apply Lemma~\ref{lem:operator-family}(ii) to conclude that
  $P \gg_u 0$.
\end{proof}

In the proof of the implication ``(i) $\Rightarrow$ (ii)'' we used
Lemma~\ref{lem:behaviour-resolvent}(ii) asserting that $\lambda
R(\lambda,A)\to P$ with respect to the operator norm in $\calL(E,D(A))$
as $\lambda \downarrow 0$. One might thus be tempted to conjecture that
$R(\phdot,A)$ is uniformly eventually strongly positive with respect to
$u$. However, a counterexample from \cite[Example~5.7]{DanersI} with
$E=C(K)$ and $u=\one$ shows that this is not the case. Although for
each $f \in E$ we have $f \ge cu$ for some $c > 0$, the problem is that,
as $f$ varies, the constant $c$ can be become arbitrarily small, even if
we require $\|f\| = 1$.

\section{Eventually strongly positive semigroups}
\label{section:semigroups-strong}
We are finally ready to turn to one of the main topics of our article and
consider eventually strongly positive semigroups. We start with the
precise definitions that we will use.
\begin{definition}
  \label{def:semigroups-strong}
  Let $(e^{tA})_{t \ge 0}$ be a real $C_0$-semigroup on a complex Banach
  lattice $E$.  Let $u \in E_+$ be a quasi-interior point.
  \begin{enumerate}[(a)]
  \item The semigroup $(e^{tA})_{t \ge 0}$ is called \emph{individually
      eventually strongly positive with respect to $u$} if for each $f
    \in E_+ \setminus \{0\}$ there is a $t_0 > 0$ such that $e^{tA}f
    \gg_u 0$ for each $t \ge t_0$.
  \item The semigroup $(e^{tA})_{t \ge 0}$ is called \emph{uniformly
      eventually strongly positive with respect to $u$} if there is a
    $t_0 > 0$ such that $e^{tA} \gg_u 0$ for each $t \ge t_0$.
  \end{enumerate}
\end{definition}
Our main characterisation of individually eventually strongly positive
semigroups is the following theorem. While we needed the domination condition
$D(A) \subseteq E_u$ in Theorem~\ref{thm:resolvents-strong}, we now
require the ``smoothing'' assumption $e^{t_0A}E \subseteq E_u$, which in
practice is usually weaker. See
Corollary~\ref{cor:semigroups-strong-differentiable} below for a
connection between the two conditions.
\begin{theorem}
  \label{thm:semigroups-strong}
  Let $(e^{tA})_{t \ge 0}$ be a real $C_0$-semigroup with $\sigma(A)
  \not= \emptyset$ on a complex Banach lattice $E$. Suppose that the
  peripheral spectrum $\sigma_{\per}(A)$ is finite and consists of poles
  of the resolvent. If $u \in E_+$ is a quasi-interior point and if
  there exists $t_0 \ge 0$ such that $e^{t_0A}E \subseteq E_u$, then the
  following assertions are equivalent:
  \begin{enumerate}[\upshape (i)]
  \item The semigroup $(e^{tA})_{t \ge 0}$ is individually eventually
    strongly positive with respect to $u$.
  \item The semigroup $(e^{t(A-\spb(A)I)})_{t \ge 0}$ is bounded,
    $\spb(A)$ is a dominant spectral value, and its associated spectral
    projection $P$ fulfils $P \gg_u 0$.
  \item The semigroup $e^{t(A-\spb(A)I)}$ converges to some operator $Q
    \gg_u 0$ with respect to the strong operator topology as $t \to
    \infty$.
  \end{enumerate}
  If assertions {\upshape (i)--(iii)} hold, then $P = Q$.
\end{theorem}

Again, this theorem was previously only known in the case $E = C(K)$ and
$u = \one$ \cite[Theorem~5.4]{DanersI}. In this case we have $E_u=E$,
hence the condition $e^{t_0A} \subseteq E_u$ is automatically satisfied;
this shows that the known result on $C(K)$-spaces is indeed a special
case of our general Theorem~\ref{thm:semigroups-strong} above.

\begin{proof}[Proof of Theorem~\ref{thm:semigroups-strong}]
  We may assume throughout that $\spb(A) = 0$.
 
  ``(i) $\Rightarrow$ (ii)'' By \cite[Theorem~7.6]{DanersI} $\spb(A) =
  0$ is a spectral value of $A$ and by \cite[Theorem~7.7(i)]{DanersI} it
  is even an eigenvalue and admits a positive eigenvector. Applying
  Lemma~\ref{lem:operator-family}(i) to the operator family $(e^{tA})_{t
    \in [0,\infty)}$ we conclude that the family $(e^{tA}|_{E_u})_{t \in
    [0,\infty)}$ is bounded in $\calL(E_u,E)$.  By assumption there
  exists $t_0 > 0$ such that $e^{t_0A}E \subseteq E_u$ and thus, due to
  the closed graph theorem, $e^{t_0A}\in\calL(E,E_u)$.  Hence, the
  operator family $(e^{tA})_{t \in [t_0,\infty)}$ is bounded in
  $\calL(E)$ and therefore also $(e^{tA})_{t \ge 0}$ is bounded in
  $\calL(E)$.
 
  As $(e^{tA})_{t \ge 0}$ is bounded the spectral bound $\spb(A) = 0$ is
  a simple pole of $R(\phdot,A)$ and so $\im P = \ker A$. 
  Theorem~\ref{thm:semigroups-asymptotic} below
  implies (under even weaker positivity assumptions) that $\spb(A)$
  is a dominant spectral value of $A$ and that $P$ is positive.  As
  $\spb(A)$ is an eigenvalue of $A$, $P$ is non-zero and thus
  Lemma~\ref{lem:operator-family}(ii) applied to the operator family
  $(e^{tA})_{t \in [0,\infty)}$ implies that $P \gg_u 0$.
 
  ``(ii) $\Rightarrow$ (iii)'' Since all spectral values of $A|_{\ker
    P}$ have strictly negative real part and since the semigroup
  $(e^{tA})_{t \ge 0}$ is bounded, it follows from
  \cite[Theorem~2.4]{MR933321} or \cite[Corollary~5.2.6]{Neerven1996}
  that $e^{tA}$ converges strongly to $0$ on $\ker P$ as $t \to \infty$.
  As $P \gg_u 0$, $0$ is a simple pole of $R(\phdot,A)$ according to
  Proposition~\ref{prop:projections-strong} and hence we have $\im P =
  \ker A$. Thus, $e^{tA}f\to Pf$ as $t \to \infty$ for each $f \in E$.
  In particular, (iii) holds with $Q = P$.
 
  ``(iii) $\Rightarrow$ (i)'' Let $f > 0$. By assumption $\lim_{t \to
    \infty}e^{tA}f = Qf$ in $E$ and clearly, $Qf$ is a fixed point of
  each operator $e^{tA}$. As $e^{t_0A}\in\calL(E,E_u)$ for some $t_0>0$
  we conclude for $t \ge t_0$ that
  \begin{displaymath}
    e^{tA}f = e^{t_0A}e^{(t-t_0)A}f \to e^{t_0A}Qf = Qf 
    \quad \text{in } E_u \quad \text{ as } t \to \infty.
  \end{displaymath}
  Since $e^{tA}f$ is real and $Qf\gg_u 0$, this implies that $e^{tA}f
  \gg_u 0$ for all sufficiently large $t$.
\end{proof}

As in \cite[Corollary~5.6]{DanersI}, the boundedness condition in
Theorem~\ref{thm:semigroups-strong}(ii) is redundant if the semigroup
$(e^{tA})_{t \ge 0}$ is eventually norm-continuous. If we assume that
$(e^{tA})_{t \ge 0}$ is a little more regular, then we can also give a
criterion to check the condition $e^{t_0A} \subseteq E_u$: Recall that a
$C_0$-semigroup $(e^{tA})_{t \ge 0}$ on a Banach space $E$ is called
\emph{eventually differentiable} if there exists $t_0 \ge 0$ such that
$e^{t_0A}E \subseteq D(A)$. In that case $e^{tA}E \subseteq D(A)$ for
all $t \ge t_0$. Note that each analytic semigroup is eventually (in
fact immediately) differentiable.

\begin{corollary}
  \label{cor:semigroups-strong-differentiable}
  Let $(e^{tA})_{t \ge 0}$ be a real, eventually differentiable
  $C_0$-semigroup with $\sigma(A) \not= \emptyset$ on a complex Banach
  lattice $E$. Suppose that the peripheral spectrum $\sigma_{\per}(A)$
  is finite and consists of poles of the resolvent. If $u \in E_+$ is a
  quasi-interior point and if there exists $n \in \bbN$ such that
  $D(A^n) \subseteq E_u$, then the following assertions are equivalent:
  \begin{enumerate}[\upshape (i)]
  \item The semigroup $(e^{tA})_{t \ge 0}$ is individually eventually
    strongly positive with respect to $u$.
  \item $\spb(A)$ is a dominant spectral value, and its associated
    spectral projection $P$ fulfils $P \gg_u 0$.
  \end{enumerate}
  \begin{proof}
    Let $t_0 \ge 0$ such that $e^{t_0A} E \subseteq D(A)$. Then
    $e^{nt_0A}E \subseteq D(A^n) \subseteq E_u$ and hence the assumptions of
    Theorem~\ref{thm:semigroups-strong} are fulfilled. The implication
    ``(i) $\Rightarrow$ (ii)'' therefore follows.
    
    Now assume that (ii) is true. To conclude from
    Theorem~\ref{thm:semigroups-strong} that (i) holds, we only have to
    show that $(e^{t(A-\spb(A))})_{t \ge 0}$ is bounded. Without loss of
    generality we assume that $\spb(A) = 0$. Since $P \gg_u 0$,
    Proposition~\ref{prop:projections-strong} tells us that $\spb(A) =
    0$ is a simple pole of the resolvent; hence, $(e^{tA})_{t \ge 0}$ is
    bounded on $\im P$. Since the semigroup is eventually
    differentiable, it is in particular eventually norm-continuous, and
    hence $\sigma(A)\cap\{z \in \bbC\colon \repart z \ge \alpha\}$ is
    bounded for every $\alpha \in \bbR$, see
    \cite[Theorem~II.4.18]{Engel2000}. Since $\spb(A) = 0$ is a dominant
    spectral value of $A$ this implies that $\spb(A|_{\ker P}) < 0$.
    Using again the eventual norm-continuity of the semigroup we
    conclude that the growth bound of $(e^{tA}|_{\ker P})_{t \ge 0}$
    is strictly negative; in particular, $(e^{tA}|_{\ker
      P})_{t \ge 0}$ is bounded.
  \end{proof}
\end{corollary}

It is interesting to note that we needed the condition $D(A) \subseteq
E_u$ to prove Theorem~\ref{thm:resolvents-strong} about
resolvents, while we only need the weaker assumption $D(A^n) \subseteq
E_u$ for some power $n \in \bbN$ in
Corollary~\ref{cor:semigroups-strong-differentiable} about (eventually
differentiable) semigroups. When we consider the bi-Laplacian with
Dirichlet boundary conditions in
Section~\ref{section:applications-strong}, this will allow us to prove
stronger results on the semigroup than on the resolvent (compare
Propositions~\ref{prop:resolvent-bi-laplacian}
and~\ref{prop:semigroup-bi-laplacian}).

We now adapt \cite[Example~5.7]{DanersI} to show that we cannot in
general drop the domination conditions $D(A) \subseteq E_u$ and
$e^{t_0A}E \subseteq E_u$ in Theorems~\ref{thm:resolvents-strong} and
\ref{thm:semigroups-strong}. Interestingly, the example was used in
\cite{DanersI} as a counterexample to a rather different question.

\begin{example}
  \label{example:domination-property}
  Let $p \in [1,\infty)$ and $E = L^p((-1,1))$.  We denote by $\one$ the
  constant function with value one and by $\varphi\colon E \to \bbC$ the
  continuous linear functional given by $\varphi(f) = \int_{-1}^1
  f(\omega)\,d\omega$. Consider the decomposition
  \begin{displaymath}
    E = \langle \one \rangle \oplus F \qquad \text{with} \qquad F:= \ker
    \varphi \text{.}
  \end{displaymath}
  Let $S \colon F \to F$ denote the reflection operator given by
  $Sf(\omega) = f(-\omega)$ for all $\omega\in (-1,1)$. As $S^2=\id_F$
  we have $\sigma(S) = \{-1,1\}$.  We define $A\in\calL(E)$ by
  \begin{displaymath}
    A := 0_{\langle \one \rangle} \oplus (- 2 \id_F - S) \text{.}
  \end{displaymath}
  Clearly, $\sigma(A)=\{0,-1,-3\}$ and using $S^2 = \id_F$, we can 
  immediately check that
  \begin{align}
    e^{tA} & = \id_{\langle \one \rangle} \oplus e^{-2t}\big( \cosh (t)
    \, \id_F - \sinh (t) \, S \big) \quad \text{and}
    \label{form_sg_counter_ex_on_l_p_space} \\
    R(\lambda,A) & = \frac{1}{\lambda} \id_{\langle \one \rangle} \oplus
    \frac{1}{(\lambda + 2)^2 - 1} \big( (\lambda + 2) \id_F - S \big)
    \text{.}
    \label{form_res_counter_ex_on_l_p_space}
  \end{align}
  for all $t \ge 0$ and all $\lambda \in \rho(A) = \bbC \setminus
  \{0,-1,-3\}$.

  Now let $P$ be the spectral projection associated with $\spb(A) =
  0$. We clearly have $P f = \frac{1}{2} \varphi(f) \cdot \one$ for all
  $f \in E$. Thus, $P$ is strongly positive with respect to $u = \one$.
  Moreover, all assumptions of Theorems~\ref{thm:resolvents-strong} and
  \ref{thm:semigroups-strong} are fulfilled, except that $D(A) = E \not
  \subseteq L^\infty((-1,1)) = E_u$ and $e^{t_0A}E = E \not \subseteq
  E_u$ for each $t_0 \ge 0$.

  Now, consider $f \in E$ given by $f(\omega) =
  (1-\omega)^{-\frac{1}{2p}}$ for all $\omega \in (-1,1)$. Note that $f$
  is bounded on $[-1,0]$, but unbounded for $\omega$ close to $1$. By
  splitting $f$ into $P f \in \langle \one \rangle$ and $(1 - P)f \in F$
  and applying the formulae \eqref{form_sg_counter_ex_on_l_p_space} and
  \eqref{form_res_counter_ex_on_l_p_space}, we see that $e^{tA} f \not
  \ge 0$ for all $t > 0$ and $R(\lambda,A)f \not \ge 0$ for all $\lambda
  > 0$, that is, both the semigroup and the resolvent are not
  individually eventually positive. In particular, they are not
  individually eventually strongly positive with respect to $u$.
\end{example}

\section{Applications of eventual strong positivity}
\label{section:applications-strong}
We shall now give some applications of the theory developed so far.
Several applications were already given in \cite[Section~6]{DanersI},
but now we have much more freedom since we are not confined to
$C(K)$-spaces. Our first two examples are concerned with bi-harmonic
operators with different boundary conditions and on different spaces.
Then we show how our results can be reformulated in the setting of a
self-adjoint operator on a Hilbert lattice, which we apply to the
Dirichlet-to-Neumann operator in two dimensions realised on
$L^2$-spaces.  Our final example is a class of Laplacians with non-local
boundary conditions.

\paragraph*{The square of the Dirichlet Laplacian}%
In \cite[Section~6.4]{DanersI} it was shown that, under sufficiently
strong regularity conditions, the negative square of the Robin Laplacian
on a bounded domain $\Omega$ of class $C^2$ generates an eventually
strongly positive semigroup on $C(\overline{\Omega})$.  However, the
negative square of the Dirichlet Laplacian $\Delta_D$ does not fit into
that framework, since it generates a $C_0$-semigroup on
$C_0(\Omega)$. Here we want to show that our theory on general Banach lattices 
naturally allows us to deal with such an operator. The Dirichlet
Laplacian is given by
\begin{displaymath}
  D(\Delta_D):=\{f\in C_0(\Omega)\colon \Delta f\in
  C_0(\Omega)\},\qquad
  \Delta_Df:=\Delta f,
\end{displaymath}
where $\Delta f$ is understood in the sense of distributions.

\begin{theorem}
  Let $\Omega\subseteq\mathbb R^n$ be a bounded domain of class
  $C^2$. On the Banach lattice $E = C_0(\Omega)$, consider the operator
  \begin{displaymath}
    D(A)=\{f\in D(\Delta_D)\colon \Delta_D f\in D(\Delta_D)\},
    \qquad Af  = -\Delta_D^2f\text{.}
  \end{displaymath}
  Let $u\in C_0(\Omega)$ be given by $u(x):=\dist(x,\partial\Omega)$.
  Then $A$ generates a holomorphic $C_0$-semigroup on $C_0(\Omega)$ of
  angle $\pi/2$ which is not positive, but individually eventually
  strongly positive with respect to $u$.
  \begin{proof}
    It is known that $\Delta_D$ generates a holomorphic $C_0$-semigroup
    of angle $\pi/2$ on $C_0(\Omega)$; see
    \cite[Theorem~2.3]{arendt:99:wrh}. We have $\sigma(\Delta_D)
    \subseteq (-\infty,0)$. Therefore, $A=-\Delta_D^2$ also generates a
    holomorphic $C_0$-semigroup of angle $\pi/2$ on $C_0(\Omega)$ as
    shown in the first part of \cite[Proposition~6.5]{DanersI}.
 
    To show that $R(0,\Delta_D)\gg_u 0$ assume that $f \in D(\Delta_D)$
    with $-\Delta_D f = g>0$. As $\Omega$ is of class $C^2$ and
    $C_0(\Omega)\subseteq L^p(\Omega)$ for all $p\in (1,\infty)$,
    classical regularity theory implies that $f\in W^{2,p}(\Omega)\cap
    C_0(\Omega)$ for all $p>n$. In particular, by standard Sobolev
    embedding theorems, $f\in C^1(\overline\Omega)$. Applying a Sobolev
    space version of the maximum principle and the strong boundary
    maximum principle we see that $\partial f/\partial\nu<0$ on the
    compact manifold $\partial\Omega$, where $\nu$ is the outer unit
    normal, see \cite{MR0223711} or \cite[Theorem~6.1]{amann:83:dss}. 
    Hence $f \gg_u 0$. It
    follows that $R(0,\Delta_D)\gg_u0$ and thus $R(0,A) =
    R(0,\Delta_D)^2 \gg_u 0$. For $\lambda \in (\spb(A),0)$ we obtain
    \begin{displaymath}
      R(\lambda,A) = \sum_{n=0}^\infty (-\lambda)^n R(0,A)^{n+1} \gg_u 0
      \text{,}
    \end{displaymath}
    so the resolvent $R(\phdot,A)$ is uniformly eventually strongly
    positive with respect to $u$.
 
    By the Sobolev embedding theorem, for $p>n$, $D(A) \hookrightarrow
    W^{2,p}(\Omega)\cap C_0(\Omega)\hookrightarrow C_0(\Omega) = E$ is
    compact, so $A$ has compact resolvent, and hence, $\spb(A)$ is a
    pole of the resolvent.  By \cite{MR614221}, we have $\partial
    u/\partial\nu<0$ on $\partial\Omega$ and hence $D(A) \subseteq
    C^1(\overline\Omega)\cap C_0(\Omega) \subseteq
    C_0(\Omega)_u$. Theorem~\ref{thm:resolvents-strong} now yields that
    the spectral projection $P$ associated with $\spb(A)$ fulfils $P
    \gg_u 0$. As $\spb(A)$ is dominant,
    Theorem~\ref{thm:semigroups-strong} finally implies that the
    semigroup $(e^{tA})_{t \ge 0}$ is individually eventually strongly
    positive with respect to $u$. That the semigroup is not positive
    follows from \cite[Proposition~2.2]{ABR90}.
  \end{proof}
\end{theorem}

\paragraph*{The bi-Laplace operator with Dirichlet boundary conditions}

Let $\Omega \subseteq \bbR^n$ be a bounded domain of class
$C^\infty$. Consider the bi-Laplace operator $A_p$ with Dirichlet
boundary conditions on $L^p(\Omega)$ ($1<p<\infty$), given by
\begin{displaymath}
  A_p\colon D(A_p):= W^{4,p}(\Omega) \cap W^{2,p}_0(\Omega)
  \to L^p(\Omega),\quad
  f\mapsto A_p f:=-\Delta^2f \text{.}
\end{displaymath}
This operator has the following properties:
\begin{proposition}
  \label{prop:bi-laplacian-on-lp}
  For $p \in (1,\infty)$ the operator $A_p$ is a closed, densely defined
  operator on $L^p(\Omega)$ having compact resolvent, and $\sigma(A_p)$
  is independent of $p \in (1,\infty)$. Moreover, the resolvent
  operators are consistent on the $L^p$-scale for $p \in (1,\infty)$.
\end{proposition}
\begin{proof}
  Clearly, $A_p$ densely defined. If $p_1 < p_2$ and
  $\lambda \in \bbC$ is an eigenvalue of $A_{p_2}$, then
  $\ker(\lambda\id - A_{p_2}) \subseteq \ker(\lambda\id - A_{p_1})$
  since $D(A_{p_2}) \subseteq D(A_{p_1})$. On the other hand,
  \cite[Corollary~2.21]{Gazzola2010} together with a simple
  bootstrapping argument shows that each function in $\ker(\lambda\id -
  A_{p_1})$ is continuous up to the boundary, hence in
  $L^{p_2}(\Omega)$, and therefore in $\ker(\lambda\id -
  A_{p_2})$. Hence, the point spectrum of $A_p$ and the corresponding
  eigenspaces do not depend on $p$.
 
  It follows from \cite[Corollary~2.21]{Gazzola2010}
  that $0 \in \rho(A_p)$. In particular, $A_p$ is closed. Since $D(A_p)$
  is compactly embedded in $L^p(\Omega)$ and since
  $\rho(A_p)\neq\emptyset$, we conclude that $A_p$ has compact
  resolvent. Therefore, $\sigma(A_p)$ consists of eigenvalues only and
  is independent of $p$ by what we have shown above.
 
  To see that the resolvent operators are consistent on the $L^p$-scale,
  let $p_1 < p_2$ and suppose $\lambda \not\in \sigma(A_{p_1}) =
  \sigma(A_{p_2})$. If $f \in L^{p_2}$, then $w = R(\lambda,A_{p_2})f$
  is a function in $D(A_{p_2}) \subseteq D(A_{p_1})$ and $(\lambda - A_{p_2})w =
  f$. Hence $(\lambda - A_{p_1})w = f$ and thus $R(\lambda,A_{p_1})$ and
  $R(\lambda,A_{p_2})$ agree on $D(A_{p_2})$.
\end{proof}

We shall consider the function $u\colon \Omega \to \bbC$, $u(x) =
\dist(x,\Omega)^2$; $u$ is a quasi-interior point of $L^p(\Omega)_+$
for every $p \in [1,\infty)$. The following result was proved by Grunau
and Sweers in \cite[Theorem~5.2]{GrSw98}.

\begin{theorem}
  \label{thm:positive-eigenfunction-bi-laplacian}
  Let $1 < p < \infty$. Suppose that $\Omega$ is sufficiently close to
  the unit ball in $\bbR^n$ in the sense of \cite[Theorem~5.2]{GrSw98}
  (where we have $m=2$). Then the eigenspace of the operator $A_p$ for
  the largest real eigenvalue is spanned by a function $v\gg_u 0$.
\end{theorem}

In \cite[Proposition~5.3]{GrSw98} Grunau and Sweers used this result to
prove that for sufficiently large $p$ the resolvent $R(\phdot,A_p)$ is
in a sense individually eventually strongly positive (though
they did not use this terminology). We now demonstrate that this result
fits into our general theory; we also do not require their assumption $p
\geq 2$. In fact, for the semigroup, we do not even need to assume
that $p>n/2$.

\begin{lemma}
  \label{lemma:spectral-bound-bi-laplacian}
  Let $p \in (1,\infty)$ and let $\Omega \in C^\infty$ be such that the
  conclusion of Theorem~\ref{thm:positive-eigenfunction-bi-laplacian}
  holds. Then $\lambda_0:=\spb(A_p)$ is a dominant spectral value of
  $A_p$ and a simple pole of $R(\phdot, A_p)$; the corresponding
  spectral projection $P$ satisfies $P \gg_u 0$.
\end{lemma}
\begin{proof}
  Since $A_2$ is self-adjoint, all of its spectral values are real.
  Proposition~\ref{prop:bi-laplacian-on-lp} thus implies that
  $\spb(A_p)$ is the largest real eigenvalue and a dominant spectral
  value of $A_p$; moreover, it is a pole of the resolvent
  $R(\phdot,A_p)$ since the resolvent is compact. According to
  Theorem~\ref{thm:positive-eigenfunction-bi-laplacian} there is an
  eigenfunction $v$ for the eigenvalue $\spb(A_p)$ such that $v \gg_u
  0$. As $\Omega$ is of class $C^\infty$, we have $v \in C^\infty
  (\overline\Omega)$ by standard regularity theory. Hence, $v$ is also
  an eigenvector of $A_2$ and thus of $A_2' = A_2$. Again since 
  $v \in C^\infty(\overline\Omega)$, we conclude that $v$
  is also an eigenfunction of $A_p'$.
  Corollary~\ref{cor:projections-strong} now yields that the spectral
  projection $P$ associated with the eigenvalue $\spb(A_p)$ of $A_p$ is
  strictly positive with respect to $u$ and that $\spb(A_p)$ is an
  algebraically simple eigenvalue; in particular, it is a simple pole of
  the resolvent.
\end{proof}

\begin{proposition}
  \label{prop:resolvent-bi-laplacian}
  Let $p \in (n/2,\infty)$ and let $\Omega \in C^\infty$ be such that the
  conclusion of Theorem~\ref{thm:positive-eigenfunction-bi-laplacian}
  holds. Then the
  resolvent $R(\phdot,A_p)$ is individually eventually strongly positive
  with respect to $u$ at the largest real eigenvalue $\lambda_0 =
  \spb(A_p)$ of $A_p$.
\end{proposition}
\begin{proof}
  By Theorem~\ref{thm:resolvents-strong}, using
  Lemma~\ref{lemma:spectral-bound-bi-laplacian} it remains to show that
  $D(A_p) \subseteq L^p(\Omega)_u$. As $p > n/2$, we know that
  $D(A_p) \subseteq W^{4,p}(\Omega) \hookrightarrow
  C^2(\overline{\Omega})$. For every $f \in D(A_p)$, the trace of $f$
  and of its weak gradient $\nabla f$ on $\partial \Omega$ are
  $0$. Thus, $f=0$ and $\nabla f=0$ on $\partial \Omega$ in the
  classical sense. Hence $D(A_p) \subseteq L^p(\Omega)_u$.
\end{proof}

The operator $A_p$ generates an analytic $C_0$-semigroup 
$(e^{tA_p})_{t \ge 0}$ on $L^p(\Omega)$ 
\cite[Theorem~5.6 on p.\,189]{Tanabe1997}. This semigroup has the following
eventual positivity property.

\begin{proposition}
  \label{prop:semigroup-bi-laplacian}
  Let $p \in (1,\infty)$ and let $\Omega \in C^\infty$ be such that the
  conclusion of Theorem~\ref{thm:positive-eigenfunction-bi-laplacian}
  holds. Then the
  semigroup $(e^{tA_p})_{t \ge 0}$ is individually eventually strongly
  positive with respect to $u$.
\end{proposition}
\begin{proof}
  The semigroup $(e^{tA_p})_{t \ge 0}$ is analytic and, using
  Lemma~\ref{lemma:spectral-bound-bi-laplacian}, by
  Corollary~\ref{cor:semigroups-strong-differentiable} it only remains
  to show that $D(A_p^n) \subseteq L^p(\Omega)_u$ for some $n \in
  \bbN$. However, we see from \cite[Corollary~2.21]{Gazzola2010} that
  $D(A_p^n) \subseteq W^{4n,p}(\Omega)$ for all $n \in \bbN$. Hence, the
  Sobolev embedding theorem yields $D(A_p^n) \subseteq
  C^2(\overline{\Omega})$ for all sufficiently large $n$. Since we also
  have $D(A_p^n) \subseteq W_0^{2,p}(\Omega)$ for all $n$, we can now
  conclude as in the proof of
  Proposition~\ref{prop:resolvent-bi-laplacian} that $D(A_p^n) \subseteq
  L^p(\Omega)_u$.
\end{proof}

It seems quite interesting that we need the assumption $p \in
(n/2,\infty)$ only for the individual eventual strong positivity of the
resolvent $R(\phdot,A_p)$, but not for the same property of the
semigroup $(e^{tA_p})_{t \ge 0}$. This is of course due to the fact that
Theorem~\ref{thm:resolvents-strong} requires the condition $D(A) \subseteq
E_u$, while Corollary~\ref{cor:semigroups-strong-differentiable} only
requires the weaker assumption $D(A^n) \subseteq E_u$ for some $n \in
\bbN$; compare also the related discussion after
Corollary~\ref{cor:semigroups-strong-differentiable}.

\paragraph*{Eventual strong positivity for self-adjoint operators on
  Hilbert lattices}
In this paragraph we reformulate our results for the special case of
self-adjoint operators on Hilbert lattices. Recall that a \emph{Hilbert
  lattice} is a Banach lattice $H$ whose norm is induced by an inner
product. For every measure space $\Omega$ the space $L^2(\Omega)$ is a
Hilbert lattice and conversely, every Hilbert lattice $H$ is
isometrically lattice isomorphic to $L^2(\Omega)$ for some measure space
$\Omega$ (see \cite[Theorem~IV.6.7 and Exercise~18(f) on p.\,303]{Schaefer1974} 
for a slightly stronger
result).

For self-adjoint operators on Hilbert lattices, our main result can be
summarised as follows.

\begin{theorem}
  \label{thm:evtl-pos-hilbert-lattice}
  Let $H$ be a complex Hilbert lattice and let $u \in H_+$ be a
  quasi-interior point. Let $A$ be a real, densely defined and self-adjoint
  operator on $H$ and assume that $\spb(A)\in\bbR$ is an isolated point
  of $\sigma(A)$. Moreover, suppose that $D(A) \subseteq H_u$. Then the
  following assertions are equivalent:
  \begin{enumerate}[\upshape (i)]
  \item The eigenvalue $\spb(A)$ is geometrically simple and has an
    eigenvector $v \gg_u 0$.
  \item The spectral projection $P$ associated with $\spb(A)$ satisfies
    $P \gg_u 0$.
  \item The resolvent $R(\phdot,A)$ is individually eventually strongly
    positive with respect to $u$ at $\spb(A)$.
  \item The semigroup $(e^{tA})_{t \ge 0}$ is individually eventually
    strongly positive with respect to $u$.
  \end{enumerate}
\end{theorem}
\begin{proof}
  As mentioned above, $H$ can be identified with $L^2(\Omega)$ for some
  measure space $\Omega$. This shows that, when restricted to the real
  part of $H$, the canonical identification $H \simeq H'$ is a lattice
  isomorphism. Under this identification, the Hilbert space adjoint of
  $A$ coincides with the Banach space dual of $A$ on the real part of
  $H$, so the equivalence ``(i) $\Leftrightarrow$ (ii)'' follows from
  Proposition~\ref{prop:projections-strong}. The equivalence ``(ii)
  $\Leftrightarrow$ (iii)'' follows from
  Theorem~\ref{thm:resolvents-strong}. Moreover, $\spb(A)$ is a dominant
  eigenvalue and since the semigroup $(e^{t(A-\spb(A))})_{t\ge 0}$ is
  analytic, the equivalence ``(ii)
  $\Leftrightarrow$ (iv)'' follows from
  Corollary~\ref{cor:semigroups-strong-differentiable}.
\end{proof}

Note that the condition $D(A) \subseteq H_u$ in 
Theorem~\ref{thm:evtl-pos-hilbert-lattice} is only needed
to show the equivalence ``(ii)
$\Leftrightarrow$ (iii)''. If we replace the condition
$D(A) \subseteq H_u$ with the weaker condition
$D(A^n) \subseteq H_u$ for some $n \in \bbN$, or even with 
$e^{t_0A} \subseteq H_u$ for some $t_0 > 0$, then the assertions
(i), (ii) and (iv) are still equivalent.

\paragraph*{The Dirichlet-to-Neumann operator in two dimensions}

In \cite[Section~6.2]{DanersI} the Dirichlet-to-Neumann operator on
$C(\Gamma)$ was analysed, where $\Gamma\subseteq\bbR^2$ is the unit
circle. Using our theory for general Banach lattices we can now consider
the more natural setting of $L^2$-spaces on more general domains.

We assume for simplicity that $\Omega \subseteq \bbR^2$ is a bounded
domain with $C^\infty$ boundary, although much of what follows still
holds under weaker assumptions. Let $\lambda \in \bbR$ be contained in
the resolvent set of the Dirichlet Laplacian $\Delta_D$ on
$L^2(\Omega)$. For $g \in L^2(\partial \Omega)$ we solve, whenever
possible, the Dirichlet problem
\begin{displaymath}
  \Delta f = \lambda f\quad \text{in $\Omega$,} \qquad
  f  = g\quad \text{on $\partial\Omega$.}
\end{displaymath}
Afterwards, we map the solution $f$ to its (distributional) normal
derivative $\frac{\partial}{\partial \nu}f$ on the boundary $\partial
\Omega$, if this is in $L^2(\partial\Omega)$.  The operator
$D_\lambda\colon g \mapsto\partial f/\partial\nu$ thus defined is called
the \emph{Dirichlet-to-Neumann} operator for the domain $\Omega$ and for
the parameter $\lambda$. For a precise definition of the
Dirichlet-to-Neumann operator $D_\lambda$ we refer the reader to
\cite{MR3146835} or \cite{arendt:12:fei}.

It can be shown that $-D_\lambda$ is a densely defined self-adjoint
operator on $L^2(\partial \Omega)$ with spectral bound $\spb(-D_\lambda)
\in \bbR$ and compact resolvent on $L^2(\partial \Omega)$; see
\cite[Proposition~2]{arendt:12:fei}.  In \cite{Daners2014} it was shown
that the semigroup $(e^{-tD_\lambda})_{t\geq 0}$ is uniformly eventually
positive, but not positive for certain $\lambda$ if $\Omega$ is the disk
in $\bbR^2$. The abstract theory developed in this paper allows us to
give a \emph{characterisation} of the semigroups $(e^{-tD_\lambda})_{t
  \ge 0}$ that are individually eventually strongly positive with
respect to $\one$.

\begin{proposition}
  \label{prop:dtn-operator}
  Let $\Omega \subseteq \bbR^2$ be a domain with $C^\infty$ boundary and
  let $\lambda \in \bbR$ be contained in the resolvent set of the
  Dirichlet Laplacian on $L^2(\Omega)$. Then the following assertions
  are equivalent:
  \begin{enumerate}[\upshape (i)]
  \item The semigroup $(e^{-tD_\lambda})_{t \ge 0}$ is individually
    eventually strongly positive with respect to $\one$.
  \item The largest real eigenvalue of $-D_\lambda$ is geometrically
    simple and admits an eigenfunction which is strongly positive with
    respect to $\one$.
  \end{enumerate}
  \begin{proof}
    It follows from \cite[Theorem~5.2]{MR3397313} that $D(D_\lambda) =
    H^1(\partial\Omega)$.  Since $\partial \Omega$ is a smooth
    one-dimensional manifold, standard embedding theorems imply
    $H^1(\partial\Omega) \subseteq C(\partial \Omega)$, the latter
    clearly being contained in $L^\infty(\partial \Omega) = L^2(\partial
    \Omega)_{\one}$.  Hence the proposition follows from
    Theorem~\ref{thm:evtl-pos-hilbert-lattice}.
  \end{proof}
\end{proposition}

\paragraph*{The Laplace operator with non-local Robin boundary
  conditions}
It is well known that the Laplace operator with Dirichlet or Neumann
boundary conditions (or more generally with Robin boundary conditions)
generates a positive $C_0$-semigroup on $L^2(\Omega)$ whenever $\Omega
\subseteq \bbR^n$ is a sufficiently regular bounded domain. We consider
the non-local Robin problem
\begin{equation}
  \label{eq:non-local-robin}
  \begin{aligned}
    -\Delta u&=f&&\text{in $\Omega$,}\\
    \frac{\partial u}{\partial\nu}+B\trace(u)&=0 &&\text{on
      $\partial\Omega$,}
  \end{aligned}
\end{equation}
where $\Omega\subseteq\bbR^2$ is a bounded Lipschitz domain, 
$B\in\calL\bigl( L^2(\partial\Omega)\bigr)$ a bounded linear operator
and $\trace\in\calL\bigl(H^1(\Omega),L^2(\partial\Omega)\bigr)$ the
trace operator.  The usual local Robin boundary condition can be
recovered as a special case, by choosing $B$ to be a multiplication
operator of the form $Bu = \beta u$ for $\beta \in
L^\infty(\partial\Omega)$.

There has been considerable interest in non-local Robin boundary
conditions in recent times. Possibly the first paper studying this problem 
in a general setting was \cite{MR2568160}. In \cite{MR3183528}, conditions 
for positivity of the semigroup are given. Earlier results on positivity and loss 
of positivity in a special case, namely a simple model of a
thermostat, appear in \cite{MR1787081}, and \cite{MR1067499} deals with
applications to Bose condensates. We discuss three examples where
eventual positivity occurs, but before we do so we look at some general
properties of \eqref{eq:non-local-robin}.

First note that the sesquilinear form corresponding to
\eqref{eq:non-local-robin} is given by
\begin{equation}
  \label{eq:non-local-robin-form}
  a(u,v):=\int_\Omega\nabla u\cdot\overline{\nabla v}\,dx
  +\langle B\trace(u),\trace(v)\rangle
\end{equation}
for all $u,v\in H^1(\Omega)$. Since $B$ and $\gamma$ are bounded
operators, the form \eqref{eq:non-local-robin-form} is bounded from
below and the values $a(u,u)$ are contained in a sector with vertex 
somewhere on 
the real line. Therefore the induced operator $-A$ generates an analytic
semigroup on $L^2(\Omega)$. As $D(A)\subseteq H^1(\Omega)$, $A$ has
compact resolvent.  We can say a bit more if $B$ is self-adjoint.
\begin{lemma}
  \label{lem:non-local_bc}
  Assume that $B\in \calL\bigl(L^2(\partial\Omega)\bigr)$ is
  self-adjoint and positive semi-definite, and that $\langle
  B\one,\one\rangle>0$. Then the operator $A$ induced by
  \eqref{eq:non-local-robin-form} on $L^2(\Omega)$ is self-adjoint and
  $[0,\infty)\subseteq\rho(-A)$.
\end{lemma}
\begin{proof}
  The form $a$ is symmetric since $B$ is self-adjoint.
  Hence $A$ is self-adjoint, too.
  By assumption $B$ is positive semi-definite. Hence $a(u,u)\geq 0$ for
  all $u\in H^1(\Omega)$ and so $(0,\infty)\subseteq\rho(-A)$. We now
  show that $A$ is injective and therefore $0\in\rho(-A)$. Assume that
  $u\in D(A)$ such that $Au=0$, that is, $0=a(u,u)=\|\nabla
  u\|_2^2+\langle B\trace(u),\trace(u)\rangle$. As $\langle
  B\trace(u),\trace(u)\rangle\geq 0$ we conclude that $\|\nabla
  u\|_2=\langle B\trace(u),\trace(u)\rangle=0$. In particular, $\nabla
  u=0$ on $\Omega$ and therefore $u=c\one$ for some constant $c\in
  \mathbb C$. Hence $\langle
  B\trace(c\one),\trace(c\one)\rangle=|c|^2\langle
  B\one,\one\rangle=0$. By assumption $\langle B\one,\one\rangle>0$, so
  $c=0$.  Therefore $u=0$, showing that $a$ is coercive and that $A$ is
  injective. 
\end{proof}

We now proceed to discuss the specific examples. The first is a simple
model of a thermostat of the form \eqref{eq:non-local-robin} with
$\Omega=(0,\pi)\subseteq\mathbb R$,
\begin{equation}
  \label{eq:B-thermostat}
  \gamma(u)=
  \begin{bmatrix}
    u(0)\\u(\pi)
  \end{bmatrix}
  \qquad\text{and}\qquad
  B:=
  \begin{bmatrix}
    0&\beta\\0&0
  \end{bmatrix},
\end{equation}
where $\beta\in\mathbb R$; note that $L^2(\partial\Omega)\simeq\mathbb
C^2$ here. An explicit calculation in \cite[Theorem~6.1]{MR1787081} or
\cite[Section~3]{MR1473861} shows that $\spb(-A)$ is a positive, dominant
and geometrically simple eigenvalue with an eigenfunction $v\gg_{\one}0$
if and only if $\beta<1/2$.  It is also shown there that the
corresponding semigroup is positive if and only if $\beta\leq 0$. The
dual interchanges the roles of the boundary points $0$ and $\pi$ and
therefore it has the same spectrum, with correspondingly reflected
eigenfunctions. By Corollary~\ref{cor:projections-strong} the spectral
projection associated with $\spb(-A)$ is strongly positive with respect
to $\one$. Regarding the
domination condition, note that $D(A)\subseteq H^1((0,1))\subseteq
C([0,1])$, so that $D(A)\subseteq E_{\one}$. Hence, applying
Corollary~\ref{cor:semigroups-strong-differentiable}, we have proved the following
theorem.
\begin{theorem}
  \label{thm:thermostat}
  Let $\Omega=(0,\pi)$ and let $B=B(\beta)$ be given as in
  \eqref{eq:B-thermostat}. Denote by $A$ the operator associated with
  the form \eqref{eq:non-local-robin-form}. Then the semigroup
  $(e^{-tA})_{t\geq 0}$ is individually eventually strongly positive with respect to
  $\one$ but not positive if and only if $\beta\in(0,1/2)$.
\end{theorem}
Let us now give a second example where we have eventual positivity
without positivity and where $B$ is symmetric; 
as we are not interested in a general theoretical
development, we will merely consider one special case, more precisely
taking $\Omega:=(0,1)$, 
\begin{equation}
  \label{eq:B-ones}
  \gamma(u)=
  \begin{bmatrix}
    u(0)\\u(1)
  \end{bmatrix}
  \qquad\text{and}\qquad
  B:=
  \begin{bmatrix}
    1 & 1 \\
    1 & 1
  \end{bmatrix}.
\end{equation}
For the operator $A$ induced by \eqref{eq:non-local-robin-form} on
$L^2((0,1))$ we have the following theorem.
\begin{theorem}\label{thm:nonlocal-robin-1}
  Let $\Omega = (0,1)$ and let $B$ be given as in \eqref{eq:B-ones}. The operator
  $A$ associated with the form \eqref{eq:non-local-robin-form} has the following 
  properties:
  \begin{enumerate}[\upshape (i)]
  \item $\sigma(-A)\subseteq(-\infty,0)$ and $R(0,-A) \gg_{\one} 0$.
  \item The semigroup $(e^{-tA})_{t \ge 0}$ is individually eventually
    strongly positive with respect to $\one$.
  \item The semigroup $(e^{-tA})_{t \ge 0}$ is not positive.
  \end{enumerate}
\end{theorem}
\begin{remark}
  The above example was considered in \cite[Example~4.5]{aw:03:dni},
  where it was claimed that the associated semigroup dominates the
  semigroup associated with the Dirichlet Laplacian on
  $L^2((0,1))$. Khalid Akhlil observed that the semigroup is not in fact
  positive (private communication), meaning the claimed domination
  cannot hold, but we see that the semigroup is at least ``almost'',
  that is, eventually, positve.
\end{remark}

\begin{proof}[Proof of Theorem~\ref{thm:nonlocal-robin-1}]
  (i) Lemma~\ref{lem:non-local_bc} implies that $\sigma(-A) \subseteq 
  (-\infty,0)$. Now let $f
  \in L^2$.  One easily verifies that the resolvent at $0$ is given by
  \begin{displaymath}
    \bigl(R(0,-A)f\bigr)(x)
    = \frac{1}{2} \int_0^x \int_y^1 f(z) \, dz \, dy
    + \frac{1}{2} \int_x^1 \int_0^y f(z) \, dz \, dy
    \quad \text{for all } x \in [0,1]
  \end{displaymath}
  for all $f\in L^2((0,1))$. If $f > 0$, then we have
  $\bigl(R(0,-A)f\bigr)(x) > 0$ for every $x \in [0,1]$. By continuity
  we conclude that even $R(0,-A)f = u \gg_{\one} 0$.
 
  (ii) For $\lambda \in (\spb(A),0)$ we have the power series expansion
  \begin{displaymath}
    R(\lambda,-A)
    = \sum_{n=0}^\infty (-\lambda)^n R(0,-A)^{n+1}
  \end{displaymath}
  and hence $R(\lambda,A)\gg_{\one} 0$. As $\sigma(A) \subseteq \bbR$
  and $A$ has compact resolvent, we conclude that
  $\spb(A)$ is a spectral value and a pole of the resolvent. Since $D(A)
  \subseteq H^1((0,1)) \subseteq C([0,1]) \subseteq (L^2)_{\one}$, we
  conclude from Theorem~\ref{thm:evtl-pos-hilbert-lattice} that
  $(e^{tA})_{t \ge 0}$ is individually eventually strongly positive with
  respect to $\one$.
 
  (iii) By the Beurling-Deny criterion \cite[Theorem~2.6]{MR2124040}, 
  the semigroup $(e^{-tA})_{t \ge 0}$ is positive if and only if 
  the form $a$ satisfies the estimate
  $a(u^+, u^-) \le 0$ for each $u \in H^1((0,1);\bbR)$. However, if we
  choose $u \in H^1((0,1);\bbR)$ such that $u(0) = -1$ and $u(1) = 1$,
  this condition is not fulfilled. Hence $(e^{-tA})_{t \ge 0}$ is not
  positive.
\end{proof}

Our third and final example of non-local boundary conditions
of the form \eqref{eq:non-local-robin} comes from Bose condensation as
studied in \cite{MR1067499,MR1696142}.  As in \cite{MR1067499} we will
consider the example of the unit disc $\Omega$ in $\bbR^2$ and a
convolution operator $B$. We express functions on $\Omega$ in terms of
polar coordinates $r\in[0,1]$ and $\theta\in(-\pi,\pi]$ and let $B$ be
defined by
\begin{equation}
  \label{eq:bose-condensation-bc}
  (Bf)(\theta)=(q*f)(\theta)
  :=\int_{-\pi}^\pi q(\theta-\varphi)f(\varphi)\,d\varphi,
\end{equation}
where $q\in L^1((-\pi,\pi))$ and $f\in L^2((-\pi,\pi))$. We identify $q$
and $f$ with $2\pi$-periodic functions on $\mathbb R$ so that the
integral in \eqref{eq:bose-condensation-bc} makes sense. We consider
conditions under which $(e^{-tA})_{t\geq 0}$, called the Schr\"odinger
semigroup, is individually eventually strongly positive but not
positive.

By Young's inequality for convolutions we have $\|Bf\|_2=\|q*f\|_2\leq
\|q\|_1\|f\|_2$ for all $f\in L^2(\partial\Omega)$ and therefore
$B\in\calL(L^2(\partial\Omega))$.  To ensure that $B$ is real, self-adjoint
and positive semi-definite we assume that 
the Fourier coefficients $q_k$ of $q$ are real and satisfy
\begin{equation}
  \label{eq:B-positive-definite-conditions}
  q_k:=\int_{-\pi}^\pi q(\varphi)e^{-ik\varphi}\,d\varphi\geq 0
\end{equation}
for all $k\in\bbZ$. Since all Fourier coefficients $q_k$
are real, we have $q(\theta)=\overline{q(-\theta)}$ for all $\theta 
\in \bbR$, a condition which is necessary and sufficient for $B$ 
to be self-adjoint.
\begin{theorem}[Bose condensation]
  \label{thm:bose-condensation}
  Let $\Omega$ be the unit disc in $\bbR^2$. Let $B$ be the convolution
  operator \eqref{eq:bose-condensation-bc} with $q\in
  L^1(\partial\Omega)$ so that \eqref{eq:B-positive-definite-conditions} is
  satisfied with $q_0>0$. Then $B$ is positive definite and the
  operator $A$ associated with \eqref{eq:non-local-robin-form} on
  $L^2(\Omega)$ has the following properties.
  \begin{enumerate}[\upshape (i)]
  \item $A$ has compact resolvent and $\spb(-A)<0$ is an algebraically
    simple eigenvalue.
  \item The spectral projection $P$ associated with $\spb(-A)$ is
    strongly positive with respect to $\one$.
  \item If in addition, $q$ is real, then $(e^{-tA})_{t\geq 0}$ is
    individually eventually strongly positive with respect to $\one$,
    but not positive.
  \end{enumerate}
\end{theorem}
\begin{proof}
  We start by showing that $B$ given by \eqref{eq:bose-condensation-bc}
  is positive semi-definite. Let $f\in L^2(\partial\Omega)$ with Fourier
  coefficients $f_k=\int_{-\pi}^\pi
  f(\theta)e^{-ik\theta}\,d\theta\in\bbC $. The convolution theorem for
  Fourier series asserts that
  \begin{displaymath}
    (Bf)(\theta)
    =(q*f)(\theta)
    =\frac{1}{2\pi}\sum_{k=-\infty}^\infty q_kf_ke^{ik\theta}
  \end{displaymath}
  in $L^2(\partial\Omega)$; see
  \cite[Section~1.7]{katznelson:04:iha}. Hence, by the orthogonality of
  $(e^{ik\theta})_{k\in\bbZ}$ in $L^2(\partial\Omega)$ and
  \eqref{eq:B-positive-definite-conditions}
  \begin{equation}
    \label{eq:B-positive-definite}
    \langle Bf,f\rangle
    =\frac{1}{(2\pi)^2}
    \Bigl\langle\sum_{k=-\infty}^\infty q_kf_ke^{ik\theta},
    \sum_{k=-\infty}^\infty f_ke^{ik\theta}\Bigr\rangle
    =\frac{1}{2\pi}\sum_{k=-\infty}^\infty q_k|f_k|^2\geq 0.
  \end{equation}
  Hence $B$ is positive semi-definite.  If we choose $f=\one$, then
  $f_0=1$ and $f_k=0$ otherwise, so $\langle
  B\one,\one\rangle=q_0$. Hence the condition $\langle
  B\one,\one\rangle>0$ is equivalent to $q_0>0$.

  (i) Having shown that $B$ is positive definite on
  $L^2(\partial\Omega)$ we deduce from Lemma~\ref{lem:non-local_bc} and
  the discussion preceding it that $A$ has compact resolvent and
  $\spb(-A)<0$.

  (ii) To show that the spectral projection associated with $\spb(A)>0$
  is strongly positive we compute the eigenvalues and
  eigenfunctions. Let $J_k$ be the Bessel functions of the first kind,
  whose properties we use freely, see for instance
  \cite[Chapter~VII]{MR0065391}. The function
  $u(r,\theta)=J_k(\sqrt{\lambda}r)e^{ik\theta}$ is a solution of
  $\Delta u+\lambda u=0$ for every $\lambda>0$; see
  \cite[Section~5.5]{MR0065391}. The values of $\lambda>0$ such that $u$
  satisfies the boundary conditions in \eqref{eq:non-local-robin} are
  eigenvalues. We require that
  \begin{multline}
    \label{eq:bose-boundary-conditions}
    \frac{\partial}{\partial\nu}u+Bu
    =\frac{\partial}{\partial r}
    J_k\bigl(\sqrt{\lambda}r\bigr)e^{ik\theta}\Big|_{r=1}
    +J_k\bigl(\sqrt{\lambda}\bigr)\int_{-\pi}^\pi
    q(\theta-\varphi)e^{ik\varphi}\,d\varphi\\
    =\sqrt{\lambda}J_k'\bigl(\sqrt{\lambda}\bigr)e^{ik\theta}+
    q_kJ_k\bigl(\sqrt{\lambda}\bigr)e^{ik\theta} =0.
  \end{multline}
  As $J_{-k}(x)=(-1)^kJ_k(x)$ we seek $\lambda>0$ such that
  \begin{equation}
    \label{eq:bose-q_k}
    \sqrt{\lambda}J_{|k|}'\bigl(\sqrt{\lambda}\bigr)
    +q_kJ_{|k|}\bigl(\sqrt{\lambda}\bigr)=0
  \end{equation}
  for some $k\in\bbZ$. Denote by $j_{k,l}$ ($l=1,2,\dots$) the positive
  zeros of $J_k$. Note that $J_{|k|}'(s)>0$ and $J_{|k|}(s)>0$ for all
  $s\in (0,j_{0,1})$ and all $k\neq 0$. Moreover, $J_0'(s)<0$ and
  $J_0(s)>0$ for all $s\in(0,j_{0,1})$. Hence, as $q_0>0$, the smallest
  possible value of $\lambda$ satisfying \eqref{eq:bose-q_k} occurs for
  $k=0$ and $\sqrt{\lambda}\in(0,j_{0,1})$. Here we use that $J_0'(0)=0$
  and $J_0(j_{0,1})=0$ so that \eqref{eq:bose-q_k} with $k=0$ has a
  unique solution $\lambda_1\in(0,j_{0,1}^2)$. Then
  $u_1(r,\theta):=J_0(\sqrt{\lambda_1}r)>0$ is the only eigenfunction
  corresponding to $\lambda_1$. In particular, as $A$ is self-adjoint,
  $\lambda_1$ is algebraically simple and $u_1\gg_{\one} 0$. Hence the
  corresponding spectral projection $P$ satisfies $P\gg_{\one} 0$
  according to Proposition~\ref{prop:projections-strong}.

  To be sure that $\lambda_1$ is the dominant eigenvalue we need to know
  that the system of eigenfunctions we have constructed is complete; we
  sketch a proof of this fact. For any given $k\in\mathbb Z$ let
  $\lambda_{kj}$ ($j\in\mathbb N$) be the positive zeros of
  \eqref{eq:bose-q_k}. Then, the functions
  $v_{kj}(r)=J_k\bigl(\sqrt{\lambda_{kj}}r\bigr)$ ($j\in\mathbb N)$ form
  an orthonormal basis in the weighted space $L^2((0,1);r)$. A
  proof of this fact for Dirichlet boundary conditions, but easily
  modified for our conditions, appears in \cite[Example~7.12 and
  Theorem~14.10]{MR923320}.  As $e^{ik\theta}$ ($k\in\mathbb Z$) is a
  complete system on the circle it follows that
  $u_{kj}(r,\theta)=v_{kj}(r)e^{ik\theta}$ ($k\in\mathbb Z$,
  $j\in\mathbb N$) is a complete system in $L^2(\Omega)$.

  (iii) If $q$ is real valued, then $B$ is real and
  $q_k=q_{-k}\in\mathbb R$ for all $k\in\mathbb N$. It follows from (ii)
  and Theorem~\ref{thm:evtl-pos-hilbert-lattice} that $(e^{-tA})_{t\geq
    0}$ is individually eventually strongly positive with respect to
  $\one$, if in addition $D(A)\subseteq E_{\one}$. To see this note that
  every solution $u\in H^1(\Omega)$ of \eqref{eq:non-local-robin} can be
  written as $u=w+v$, where
  \begin{equation}
    \label{eq:bose-w}
    -\Delta w=f\quad\text{in $\Omega$},
    \qquad\frac{\partial w}{\partial\nu}
    =0\quad\text{on $\partial\Omega$}.
  \end{equation}
  and
  \begin{equation}
    \label{eq:bose-v}
    -\Delta v=0\quad\text{in $\Omega$},
    \qquad\frac{\partial v}{\partial\nu}
    =-Bu\quad\text{on $\partial\Omega$}.
  \end{equation}
  By standard regularity theory, $w\in H^2(\Omega)$. As $Bu\in
  L^2(\Omega)$ it follows from \cite{MR598688} that $v\in
  H^{3/2}(\Omega)$. Hence $u\in H^{3/2}(\Omega)$ and by the usual
  Sobolev embedding theorems we conclude $D(A)\subseteq
  H^{3/2}(\Omega)\hookrightarrow C(\bar\Omega)\subseteq E_{\one}$, as
  required.

  To show that $(e^{-tA})_{t\geq 0}$ is not positive we use the
  Beurling-Deny criterion which states that $(e^{-tA})_{t\geq 0}$ is
  positive if and only if $a(u^+,u^-)\leq 0$ for all $u\in H^1(\Omega)$,
  see \cite[Theorem~2.6]{MR2124040}. We will show that the criterion is
  violated for the harmonic function $u_m(r,\theta):=r^m\sin m\theta$
  for some choice of $m\in\mathbb N$. Similarly as in
  \cite[Proposition~4.8]{Daners2014}, an explicit calculation shows that
  \begin{displaymath}
    u_1^+(1,\theta)
    =\frac{1}{4i}(e^{i\theta}+e^{-i\theta})
    +\frac{1}{\pi}\sum_{k=-\infty}^\infty\frac{1}{(2k)^2-1}e^{2ik\theta}
  \end{displaymath}
  in $L^2(\partial\Omega)$. As $u_1^-=u_1^+-u_1$ and
  $u_m^\pm(1,\theta)=u_1^\pm(1,m\theta)$ we see that
  \begin{equation}
    \label{eq:fourier-expansion_u_m}
      u_m^\pm(1,\theta) =\pm\frac{1}{4i}(e^{im\theta}-e^{-im\theta})
      +\frac{1}{\pi}\sum_{k=-\infty}^\infty\frac{1}{(2k)^2-1}e^{2ikm\theta}
  \end{equation}
  Applying the convolution theorem for Fourier series as before we see
  that
  \begin{equation}
    \label{eq:Bu-plus}
    \bigl(B\gamma(u_m^+)\bigr)(\theta)
    =\bigl(q*\gamma(u_m^+)\bigr)(\theta)
    =\frac{1}{4i}(q_me^{im\theta}-q_{-m}e^{-im\theta})
    +\frac{1}{\pi}\sum_{k=-\infty}^\infty
    \frac{q_{2km}}{(2k)^2-1}e^{2ikm\theta}.
  \end{equation}
  Using \eqref{eq:fourier-expansion_u_m}, \eqref{eq:Bu-plus}, the
  orthogonality of $(e^{ik\theta})_{k\in\bbZ}$ and the fact that
  $q_k=q_{-k}$ we deduce that
  \begin{equation}
    \label{eq:beurling-deny-computation}
    \begin{split}
      a(u_m^+,u_m^-) &=\int_\Omega\nabla u_m^+\nabla
      u_m^-\,dx+\bigl\langle
      B\trace(u_m^+),\trace(u_m^-)\bigr\rangle\\
      &=0+2\pi\left(-\frac{q_m+q_{-m}}{16}+\frac{1}{\pi^2}
        \sum_{k=-\infty}^\infty\frac{q_{2km}}{\bigl((2k)^2-1\bigr)^2}\right)\\
      &=\frac{2}{\pi}q_0-\frac{\pi}{4}q_m
      +\frac{4}{\pi}\sum_{k=1}^\infty
      \frac{q_{2km}}{\bigl((2k)^2-1\bigr)^2}
    \end{split}
  \end{equation}
  for all $m\geq 1$. Since $q_k\to 0$ as $k\to\infty$ by the
  Riemann-Lebesgue lemma and $q_0>0$ we can choose $m\geq 1$ such that
  $2q_0/\pi-q_m\pi/4>0$. Since $q_{2km} \geq 0$ for all $k \geq 1$ by
  assumption, we conclude that $a(u_m^+,u_m^-)>0$ for this $m$.  This
  violates the Beurling-Deny criterion for the positivity of the
  semigroup generated by $-A$ and therefore $(e^{-tA})_{t\geq 0}$ is not
  positive.
\end{proof}

\section{Asymptotically positive resolvents}
\label{section:resolvents-asymptotic}
In \cite{DanersI} and the preceding sections we considered semigroups
and resolvents which were, in some appropriate sense, eventually
strongly positive.  Nevertheless, the results presented have some
limitations.

First, for our characterisations of eventual positivity, we have always
required a domination or smoothing condition such as $D(A) \subseteq
E_u$, cf.~Theorems~\ref{thm:resolvents-strong} and~\ref{thm:semigroups-strong}
and Corollary~\ref{cor:semigroups-strong-differentiable}. As 
Example~\ref{example:domination-property} illustrates, such conditions
cannot in general be dropped.

Second, the relationship between individual and uniform eventual
positivity properties is not clear. We showed in \cite[Examples~5.7 and
5.8]{DanersI} that it is essential to distinguish between individual and
uniform eventual positivity, even under rather strong regularity and
compactness assumptions.

Third, one might suspect that in certain applications a form of eventual
positivity could occur which cannot be described in terms of
\emph{strong} positivity. At first glance the following notions seem to
be appropriate to describe such a more general behaviour: we recall from
\cite[Section~7]{DanersI}) that a $C_0$-semigroup $(e^{tA})_{t \ge 0}$
on a complex Banach lattice $E$ is called \emph{individually eventually
  positive} if for every $f \in E_+$ there is a $t_0 \ge 0$ such that
$e^{tA}f \in E_+$ whenever $t \ge t_0$. Similarly, if $A$ is a closed
operator on $E$ and $\lambda_0$ is either $-\infty$ or a spectral value
of $A$ in $\bbR$, then we call the resolvent $R(\phdot,A)$ on $E$
\emph{individually eventually positive at $\lambda_0$} if there exists
$\lambda_2 > \lambda_0$ with the following properties:
$(\lambda_0,\lambda_2] \subseteq \rho(A)$ and for every every $f \in E_+$
there exists $\lambda_1 \in (\lambda_0,\lambda_2]$ such that
$R(\lambda,A)f \in E_+$ for all $\lambda \in
(\lambda_0,\lambda_1]$. Unfortunately, it turns out that these eventual
positivity properties are difficult to characterise. This is
demonstrated by the following two examples.

\begin{examples}
  \label{examples:ind-evtl-pos-is-complicated}
  (a) If $\spb(A)$ is a dominant spectral value and the associated
  spectral projection $P$ is positive, then $(e^{tA})_{t\geq 0}$ is not
  necessarily individually eventually positive. Indeed, consider the
  linear operator on $\bbC^3$ given by
  \begin{displaymath}
    A:=
    \begin{bmatrix}
      0 & 0 & 0 \\
      0 & -1 & -1 \\
      0 & 1 & -1
    \end{bmatrix} \text{.}
  \end{displaymath}
  Then $\sigma(A) =
  \{0,i-1,-i-1\}$ and the spectral projection $P$ associated with
  $\spb(A) = 0$ is the projection onto the first component. In
  particular $\spb(A)$ is dominant and $P$ is positive.  However, the
  semigroup
  \begin{displaymath}
    e^{tA}=
    \begin{bmatrix}
      1&0&0\\
      0&e^{-t}\cos t & - e^{-t}\sin t \\
      0&e^{-t}\sin t & e^{-t}\cos t
    \end{bmatrix} \text{.}
  \end{displaymath}
  is not eventually positive. Yet, we observe that the orbits spiral
  towards the $x$-axis, so the distance to the positive cone approaches
  zero for every positive initial condition.
  
  (b) Even if $\spb(A)$ is a dominant spectral value, individual
  eventual positivity of the resolvent at $\spb(A)$ is in general not
  equivalent to individual eventual positivity of the semigroup. Indeed,
  let $E = \bbC^n$ and let $A\in\calL(E)$ with $\spb(A) = 0$ such that
  $(e^{tA})_{t \ge 0}$ is bounded and eventually positive, but not
  positive (such semigroups exist in all dimensions $n \geq 3$, 
  cf.~\cite[Remark~5.3(a)]{DanersI}). Then $\{\lambda > 0\colon 
  R(\lambda,A) \not \ge 0\}$ is non-empty and open in $(0,\infty)$. We 
  choose an arbitrary element $\lambda_0$ from this set.

  Now, consider the operator $B := (A-\lambda_0I) \oplus 0$ on $E \oplus
  \bbC$. Then $0$ is a dominant spectral value of $B$ and a simple pole
  of its resolvent. Clearly, $B$ generates an eventually positive
  semigroup on $E \times \bbC$ with $\spb(B)=0$.  However, the resolvent
  of $B$ is not eventually positive at $0$, since $R(\lambda,B)|_{E} =
  R(\lambda_0 + \lambda, A) \not\geq 0$ for small $\lambda > 0$.
\end{examples}

In \cite[Example~8.2]{DanersI} the reader can find another example of a
$C_0$-semigroup $(e^{tA})_{t \ge 0}$ (on an infinite dimensional Banach
lattice) which is uniformly eventually positive but whose resolvent is
not individually eventually positive at $\spb(A)$. However, in that
example the semigroup is nilpotent and therefore $\spb(A) = -\infty$.

For the reasons described above it seems appropriate to introduce yet another
concept of eventual positivity which does not exhibit the above
mentioned disadvantages. An indication of what concept this should be 
can be found in Examples~\ref{examples:ind-evtl-pos-is-complicated}. 
Observe that in Example~(a) the semigroup is
not eventually positive, but its ``negative part'' tends to $0$ as time
evolves. Similarly, in Example~(b) the resolvent is
not eventually positive but, despite having a pole in $0$, its
``negative part'' remains bounded as $\lambda \downarrow 0$.  These
observations motivate Definitions~\ref{def:resolvents-asymptotic},
\ref{def:resolvents-asymptotic-bounded} and
\ref{def:semigroups-asymptotic} below.
 
Recall that for every element $f$ of a Banach lattice $E$, we denote by
$\distPos{f}$ the distance of $f$ to the positive cone $E_+$ as defined
in \eqref{eq:def-distPos}.
\begin{definition}
  \label{def:resolvents-asymptotic}
  Let $A$ be a closed linear operator on a complex Banach lattice
  $E$. Suppose that $\lambda_0\in\mathbb R$ is a spectral value of $A$
  such that $(\lambda_0,\lambda_0+\delta] \subseteq \rho(A)$
  for some $\delta > 0$ and such that $R(\phdot,A)$ satisfies the
  Abel-type growth condition
  \begin{equation}
    \label{form:growth-condition}
    \limsup_{\lambda \downarrow \lambda_0}
    \|(\lambda-\lambda_0)R(\lambda,A)\|
    <\infty\text{.}
  \end{equation}
  \begin{enumerate}[\upshape (a)]
  \item The resolvent $R(\phdot,A)$ is called \emph{individually
      asymptotically positive at $\lambda_0$} if for each $f \ge 0$ we
    have $(\lambda - \lambda_0)\distPos{R(\lambda,A)f} \to 0$ as
    $\lambda \downarrow \lambda_0$.
  \item The resolvent $R(\phdot,A)$ is called \emph{uniformly
      asymptotically positive at $\lambda_0$} if for each $\varepsilon
    > 0$ there is a $\lambda_1 > \lambda_0$ with the following
    properties: $(\lambda_0, \lambda_1] \subseteq \rho(A)$ and $(\lambda -
    \lambda_0) \distPos{R(\lambda,A)f} \le \varepsilon \|f\|$ for all $f
    \in E_+$ and all $\lambda \in (\lambda_0,\lambda_1]$.
  \end{enumerate}
\end{definition}

Note that, in contrast to Definition~\ref{def:resolvents-strong}, we do
not allow for the case $\lambda_0 = -\infty$ here, since in this case
the growth condition \eqref{form:growth-condition} does not make sense.
Let us also introduce a somewhat stronger refinement of the above
definitions.

\begin{definition}
  \label{def:resolvents-asymptotic-bounded}
  Let $A$ be a closed linear operator on a complex Banach lattice
  $E$. Suppose that $\lambda_0\in\mathbb R$ is a spectral value of $A$
  such that $(\lambda_0, \lambda_0 + \delta] \subseteq
  \rho(A)$ for some $\delta > 0$ and such that $R(\phdot,A)$
  satisfies \eqref{form:growth-condition}.
  \begin{enumerate}[\upshape (a)]
  \item The resolvent $R(\phdot,A)$ is called \emph{individually
      asymptotically positive of bounded type at $\lambda_0$} if there
      exists $\lambda_1 > \lambda_0$ with the following properties:
      $(\lambda_0, \lambda_1] \subseteq \rho(A)$ and
      the set $\{\distPos{R(\lambda,A)f}\colon \lambda \in (\lambda_0,
      \lambda_1]\}$ is bounded for every $f \in E_+$.
  \item The resolvent $R(\phdot,A)$ is called \emph{uniformly
      asymptotically positive of bounded type at $\lambda_0$} if there
    exist $\lambda_1 > \lambda_0$ and $K \ge 0$ with the following
    properties: $(\lambda_0,\lambda_1] \subseteq \rho(A)$ and
    $\distPos{(R(\lambda,A)f)} \le K\|f\|$ for all $f \in E_+$ and all
    $\lambda \in (\lambda_0, \lambda_1]$.
  \end{enumerate}
\end{definition}

Clearly, if the resolvent of $R(\phdot,A)$ is individually
asymptotically positive of bounded type, then it is also individually
asymptotically positive and the same observation also holds
for the uniform properties.

We could also define asymptotic negativity of the resolvent
from the left just as we defined eventual strong negativity of the resolvent in
Definition~\ref{def:resolvents-strong-negative}. However, this
definition would not lead to any fundamentally new concepts, and it does
not seem to have applications similar to the anti-maximum principle that
we considered in Example~\ref{ex:anti-maximum-principle}. 
We shall therefore not discuss it in detail.

Note that in Definition~\ref{def:resolvents-asymptotic-bounded}(a)
$\lambda_1>\lambda_0$ can always be chosen independently of $f$, but
with respect to $f$ in the unit ball, no uniform upper bound for
\begin{displaymath}
  \bigl\{\distPos{R(\lambda,A)f}\colon \lambda \in (\lambda_0,
  \lambda_1]\bigr\}
\end{displaymath}
is guaranteed. In contrast, (b) requires the existence of such a uniform
bound.

\begin{remark}
  Note that if $\lambda_0\in\sigma(A)$ is a pole of the resolvent, then
  the growth condition \eqref{form:growth-condition} is fulfilled if and
  only if $\lambda_0$ is a simple pole.
\end{remark}
 
To state our main theorem of this section it will be useful to introduce
the notion of asymptotic positivity not only for resolvents of unbounded
operators, but also for powers of bounded operators. We recall that a
bounded operator $T$ is called power bounded if the sequence
$(T^n)_{n\in\bbN_0}$ is bounded with respect to the operator norm.

\begin{definition}
  \label{def:operators-asymptotic}
  Let $T$ be a bounded linear operator on a complex Banach lattice $E$
  with $\spr(T) > 0$ such that $\frac{T}{\spr(T)}$ is power bounded.
  \begin{enumerate}[\upshape (a)]
  \item We call $T$ \emph{individually asymptotically positive} if
    $\distPos{\frac{T^n}{\spr(T)^n}f} \to 0$ as $n \to \infty$ for every
    $f\in E_+$.
  \item We call $T$ \emph{uniformly asymptotically positive} if for each
    $\varepsilon > 0$ there is an $n_0 \in \bbN_0$ such that such that
    $\distPos{\frac{T^n}{r(T)^n}f} \le \varepsilon \|f\|$ for every $n
    \ge n_0$ and every $f \in E_+$.
  \end{enumerate}
\end{definition}

We can now state our main theorem on asymptotically positive resolvents.

\begin{theorem}
  \label{thm:resolvents-asymptotic}
  Let $A$ be a closed linear operator on a complex Banach lattice $E$
  and suppose that $\lambda_0 \in \sigma(A) \cap \bbR$ is a simple pole
  of $R(\phdot,A)$.  Then the following assertions are equivalent:
  \begin{enumerate}[\upshape (i)]
  \item The spectral projection $P$ associated with $\lambda_0$ is
    positive, that is, $P \ge 0$.
  \item The resolvent $R(\phdot,A)$ is individually asymptotically
    positive at $\lambda_0$.
  \item The resolvent $R(\phdot,A)$ is uniformly asymptotically positive
    of bounded type at $\lambda_0$.
  \end{enumerate}
  If $\lambda_0=\spb(A)$, then for every $\lambda > \spb(A)$
  the operator $R(\lambda,A) \spr(R(\lambda,A))^{-1}$
  is power bounded and the above assertions {\upshape (i)}--{\upshape (iii)} 
  are also equivalent to:
  \begin{enumerate}[\upshape (i)]\setcounter{enumi}{3}
  \item There is a $\lambda > \spb(A)$ such that the operator
    $R(\lambda,A)$ is individually asymptotically positive.
  \item For each $\lambda > \spb(A)$ the operator $R(\lambda,A)$ is
    uniformly asymptotically positive.
  \end{enumerate}
\end{theorem}
A few remarks are in order. First, note that in contrast to
Theorem~\ref{thm:resolvents-strong} we now \emph{assume} $\lambda_0$ to
be a simple pole. Indeed, we cannot expect an asymptotically positive
resolvent to have a simple pole in $\lambda_0$ automatically. To see
this simply consider a two-dimensional Jordan block with eigenvalue
$0$. Second, note that we did not need any domination assumption such as
$D(A) \subseteq E_u$. Hence the theorem is applicable in a wider range
of situations. Third, the above theorem yields the desired equivalence
between individual and uniform eventual behaviour which is not true 
for eventual (strong) positivity.
\begin{proof}[Proof of Theorem~\ref{thm:resolvents-asymptotic}]
  We may assume that $\lambda_0 = 0$. We shall prove (i) $\Rightarrow$ (iii)
  $\Rightarrow$ (ii) $\Rightarrow$ (i) and (i) $\Rightarrow$ (v)
  $\Rightarrow$ (iv) $\Rightarrow$ (i).
 
  ``(i) $\Rightarrow$ (iii)'' Choose $\varepsilon > 0$ sufficiently
  small that the closed punctured disk of radius $\varepsilon$ around $0$ is
  contained in $\rho(A)$. Then
  \begin{displaymath}
    K := \sup_{\lambda \in (0,\varepsilon)} \|R(\lambda,A)|_{\ker P}\|
      < \infty.
  \end{displaymath}
  Moreover, we have $\lambda R(\lambda,A)P=P$ for all $\lambda \in
  (0,\varepsilon)$ since $0$ is a simple pole. 
  Thus, using that $P\geq 0$, for every $f\geq 0$ and
  every $\lambda \in (0,\varepsilon)$
  \begin{displaymath}
    \distPos{R(\lambda,A)f}
    \leq\distPos{\lambda^{-1}Pf}+\|R(\lambda,A)(f-P f)\|
    \leq 0+K\|f-Pf\|\le K \|\id - P\| \|f\|.
  \end{displaymath}
  Clearly (iii) implies (ii). To see the implication ``(ii)
  $\Rightarrow$ (i)'' recall that because $\lambda_0$ is a simple pole
  of the resolvent we have $\lambda R(\lambda,A)\to P$ in $\calL(E)$ as
  $\lambda\downarrow\lambda_0$ . Hence for $f \ge 0$ we have
  \begin{math}
    \distPos{Pf} = \lim_{\lambda \downarrow 0} \distPos{\lambda
      R(\lambda,A)f}= 0
  \end{math}
  and so $P \ge 0$ as claimed.
 
  From now on, assume that $\spb(A) = \lambda_0 = 0$. If
  $\lambda > 0$, then the 
  operator $R(\lambda,A)\spr(R(\lambda,A))^{-1} = \lambda
  R(\lambda,A)$ is power bounded since its powers converge to $P$ with 
  respect to the operator norm according to 
  Lemma~\ref{lem:behaviour-resolvent}(ii). Hence, the notions of individual
  and uniform asymptotic positivity are well-defined for $R(\lambda,A)$.
  
  ``(i) $\Rightarrow$ (v)'' Let $\lambda > 0$. It is sufficient to show
  that $\lambda R(\lambda,A)$ is uniformly asymptotically positive. Note
  that $\spr\bigl(\lambda R(\lambda,A)\bigr)=1$. Moreover, as
  $\spb(A)=0$ is a simple pole, $(\lambda R(\lambda,A))^n \to P$ in
  $\calL(E)$ as $n \to \infty$, see
  Lemma~\ref{lem:behaviour-resolvent}(ii). Thus, given $\varepsilon > 0$
  there is an $n_0 \in \bbN_0$ such that $\|(\lambda R(\lambda,A))^n -
  P\| \le \varepsilon$ for all $n \ge n_0$.  Using that $P\geq 0$ we see
  that
  \begin{displaymath}
    \distPos{(\lambda R(\lambda,A))^nf}
    \leq \|(\lambda R(\lambda,A))^n-P\|\|f\| + \distPos{Pf}
    \leq \varepsilon \|f\|
    \text{.}
  \end{displaymath}
  for all $f \in E_+$ and all $n \ge n_0$. Hence $\lambda R(\lambda,A)$
  is uniformly asymptotically positive.
  
  Clearly, (v) implies (iv). To show that (iv) implies (i), let $f \ge
  0$ and observe that
  \begin{displaymath}
    \distPos{Pf} = \lim_{n \to \infty} \distPos{(\lambda
      R(\lambda,A))^nf} = 0 \text{.}
  \end{displaymath}
  Hence, $P \ge 0$ as claimed.
\end{proof}

In Proposition~\ref{prop:projections-strong} results about
\emph{strongly} positive spectral projections are given. In the setting
of asymptotic positivity, we are instead interested in projections which
are merely positive. Hence, the following proposition and its corollary
can sometimes be useful:

\begin{proposition}
  \label{prop:projections-positive}
  Let $E$ be a (real or complex) Banach lattice and let $P$ be a
  projection. If $\im P$ is one-dimensional and if $\im P$ and $\im P'$
  contain positive non-zero vectors, then $P$ is positive.
\end{proposition}
\begin{proof}
  Since $\im P$ is one-dimensional, so is $\im P'$;
  c.f.~\cite[Section~III.6.6]{Kato1976}.  Let $0 < u \in \im P$ and $0
  < \varphi \in \im P'$. We can find a vector $\psi \in E'$ such that
  $\langle \psi, u \rangle = 1$. Hence, we also have $\langle P'\psi, u
  \rangle = 1$.  Since $\im P'$ is one-dimensional, the vector $\varphi$ is a
  non-zero scalar multiple of $P'\psi$. Thus we have $\langle \varphi, u \rangle \not=
  0$ and hence $\langle \varphi, u\rangle > 0$. After an appropriate
  rescaling of $u$ we may assume that $\langle \varphi, u \rangle = 1$.
  Since $\im P$ is spanned by $u$, one now immediately computes that $Pf
  = \langle \varphi,f \rangle u$ for every $f \in E$. Hence, $P$ is positive.
\end{proof}

\begin{corollary}
  \label{cor:projections-positive}
  Let $E$ be a complex Banach lattice, let $A$ be a closed, densely
  defined operator on
  $E$ and let $\lambda_0\in\sigma(A)$ be a simple pole of the
  resolvent. Assume that $\ker(\lambda_0\id - A)$ is one-dimensional and
  contains a non-zero, positive vector and assume that
  $\ker(\lambda_0\id - A')$ contains a non-zero, positive functional.
  Then the spectral projection $P$ associated with $\lambda_0$ is
  positive.
\end{corollary}
\begin{proof}
  Since $\lambda_0$ is a simple pole of the resolvent, $\im P$ coincides
  with $\ker(\lambda_0\id - A)$ and is thus one-dimensional. The
  assertion now follows from
  Proposition~\ref{prop:projections-positive}.
\end{proof}
Proposition~\ref{prop:projections-strong} gives conditions under which
$\lambda_0$ is a first-order pole. However, the assumptions of
this proposition are very
strong if we are only interested in positive projections, and therefore
the following proposition should be useful in situations where we do not
know \emph{a priori} whether or not $\lambda_0$ is a first-order pole.

\begin{proposition}
  \label{prop:projection-positive-with-one-strictly-positive-vector}
  Let $E$ be a complex Banach lattice, let $A$ be a closed, densely
  defined operator on $E$ and let $\lambda_0 \in \sigma(A)$ be a
  pole of the resolvent. Assume that $\ker(\lambda_0\id - A)$ is
  one-dimensional and that both $\ker(\lambda_0\id - A)$ and
  $\ker(\lambda_0\id - A')$ contain positive, non-zero
  vectors. Furthermore assume that at least one of the following two
  assumptions is fulfilled:
  \begin{enumerate}[\upshape (a)]
  \item $\ker(\lambda_0\id - A)$ contains a quasi-interior point of
    $E_+$.
  \item $\ker(\lambda_0\id - A')$ contains a strictly positive
    functional.
  \end{enumerate}
  Then $\lambda_0$ is an algebraically simple eigenvalue of $A$ (in
  particular, a first-order pole of the resolvent $R(\phdot,A)$) and the
  corresponding spectral projection is positive.
\end{proposition}
\begin{proof}
  We may assume that $\lambda_0 = 0$. Since $0$ is a geometrically
  simple eigenvalue of $A$ by assumption, we only have to prove that it
  is a first-order pole of the resolvent in order to obtain that it is
  algebraically simple. We assume for a contradiction that $\lambda_0$
  is not a first-order pole, i.e.~there is an element $f \in \ker(A^2)
  \setminus \ker A$.
  
  Let $v \in \ker A$ and $\varphi \in \ker A'$.  Since $Af \in \ker A
  \setminus \{0\}$ and $\ker A$ is one-dimensional, we have $\alpha Af =
  v$ for some $\alpha \in \bbC$. We thus have
  \begin{displaymath}
    \langle \varphi, v \rangle = \alpha \langle \varphi, Af \rangle = 0.
  \end{displaymath}
  for all $v \in \ker A$ and all $\varphi \in \ker A'$.
  
  (a) Now assume that (a) is fulfilled. Then there exists a functional
  $0 < \varphi \in \ker A'$ and a quasi-interior point $v \in E_+$ which is
  contained in $\ker A$. For such elements we cannot have $\langle \varphi,
  v\rangle = 0$, so we have arrived at a contradiction.
  
  (b) If (b) is true, then there is a vector $0 < v \in \ker A$ and a
  strictly positive functional $\varphi \in \ker A'$. Again we cannot have
  $\langle \varphi, v \rangle = 0$, and thus we obtain a contradiction.
  
  We have proved that $\lambda_0$ is an algebraically simple eigenvalue of
  $A$.  Corollary~\ref{cor:projections-positive} now implies that the
  corresponding spectral projection is positive.
\end{proof}

\section{Asymptotically positive semigroups}
\label{section:semigroups-asymptotic}

In this section we characterise asymptotically positive semigroups. We
begin with the major definitions.

\begin{definition}
  \label{def:semigroups-asymptotic}
  Let $(e^{tA})_{t \ge 0}$ be a $C_0$-semigroup on a complex Banach
  lattice $E$ with $\spb(A) > -\infty$ and assume that
  $(e^{t(A-\spb(A))})_{t \ge 0}$ is bounded.
  \begin{enumerate}[\upshape (a)]
  \item The semigroup $(e^{tA})_{t \ge 0}$ is called \emph{individually
      asymptotically positive} if for every $f\geq 0$ we have
    $\distPos{e^{t(A-\spb(A))}f} \to 0$ as $t \to \infty$.
  \item The semigroup $(e^{tA})_{t \ge 0}$ is called \emph{uniformly
      asymptotically positive} if for every $\varepsilon > 0$ there is a
    $t_0 \ge 0$ such that $\distPos{e^{t(A-\spb(A))}f} \le \varepsilon
    \|f\|$ for all $t\ge t_0$ and all $f\in E_+$.
  \end{enumerate}
\end{definition}

Before proceeding, let us first note the following simple density
condition for individual asymptotic positivity.
Its proof is a simple $2\varepsilon$-argument.

\begin{proposition}
  \label{prop:density-semigroup}
  Let $(e^{tA})_{t \ge 0}$ be a $C_0$-semigroup on a complex Banach
  lattice $E$ and suppose that $(e^{t(A - \spb(A))})_{t \ge 0}$ is
  bounded.  Suppose that $D \subseteq E_+$ is dense in $E_+$ and that
  $d_+(e^{t(A-\spb(A))}g) \to 0$ as $t \to \infty$ for all $g \in
  D$. Then $(e^{tA})_{t \ge 0}$ is individually asymptotically positive.
\end{proposition}

We now state our main theorem which characterises asymptotic
positivity. In contrast to Theorem~\ref{thm:resolvents-asymptotic} on
resolvents we have to be a bit more careful here concerning the
equivalence between the statements on individual and uniform 
asymptotic positivity.

\begin{theorem}
  \label{thm:semigroups-asymptotic}
  Let $(e^{tA})_{t \ge 0}$ be a $C_0$-semigroup on a complex Banach
  lattice $E$, $\spb(A) > -\infty$ and suppose that
  $(e^{t(A-\spb(A))})_{t \ge 0}$ is bounded. Assume furthermore that
  $\sigma_{\per}(A)$ is non-empty and finite and consists of poles of
  the resolvent. Then the following assertions are equivalent:
  \begin{enumerate}[\upshape (i)]
  \item $\spb(A)$ is a dominant spectral value of $A$ and the associated
    spectral projection $P$ is positive.
  \item The semigroup $(e^{tA})_{t \ge 0}$ is individually
    asymptotically positive.
  \item The operators $e^{t(A-\spb(A))}$ converge strongly to a positive
    operator $Q$ as $t \to \infty$.
  \end{enumerate}
  If assertions {\upshape (i)-(iii)} are fulfilled, then $P = Q$.  If
  $(e^{t(A-\spb(A))})_{t \ge 0}$ is uniformly exponentially stable on
  the spectral space associated with $\sigma(A) \setminus
  \sigma_{\per}(A)$, then {\upshape (i)-(iii)} are equivalent to
  \begin{enumerate}[\upshape (i)]\setcounter{enumi}{3}
  \item The semigroup $(e^{tA})_{t \ge 0}$ is uniformly asymptotically
    positive.
  \end{enumerate}
\end{theorem}

We point out that some of the assumptions of
Theorem~\ref{thm:semigroups-asymptotic} are automatically fulfilled if
the semigroup is eventually norm continuous as the following remark shows.

\begin{remark}
  \label{rem:eventually-norm-continuous-semigroup}
  Let $(e^{tA})_{t \ge 0}$ be an eventually norm continous
  $C_0$-semigroup on a complex Banach space $E$ and assume that
  $\spb(A) > -\infty$. Then the following assertions are true:

  (i) The peripheral spectrum $\sigma_{\per}(A)$ is non-empty and
  compact.

  This follows from the fact that for an eventually norm continuous
  semigroup the set $\{\lambda \in \sigma(A)\colon \repart \lambda \ge r\}$ 
  is compact for every $r<\spb(A)$; see \cite[Theorem~II.4.18]{Engel2000}.

  (ii) If $\sigma_{\per}(A)$ consists of poles of the resolvent, then it
  is finite and $(e^{t(A-\spb(A))})_{t \ge 0}$ is uniformly
  exponentially stable on the spectral space $E_1$ associated with
  $\sigma(A) \setminus \sigma_{\per}(A)$.

  Indeed, assume that $\spb(A)=0$. If $\sigma_{\per}(A)$ consists of
  poles, then due to the compactness, $\sigma_{\per}(A)$ must be finite
  and isolated from the rest of the spectrum. It follows from
  \cite[Theorem~II.4.18]{Engel2000} that $\spb(A|_{E_1})<0$. By the eventual
  norm-continuity $(e^{tA}|_{E_1})_{t \ge 0}$ is uniformly exponentially
  stable; see \cite[Corollary~IV.3.11 and Theorem~V.1.10]{Engel2000}.

  (iii) If $\sigma_{\per}(A)$ consists of poles of the resolvent, then
  $(e^{t(A-\spb(A))})_{t \ge 0}$ is bounded if and only if all these
  poles are of first order.

  To see this assume that $\spb(A)=0$. If $(e^{tA})_{t \ge 0}$ is
  bounded, then
  \begin{math}
    \sup_{\repart \lambda > 0} \|\repart \lambda\, R(\lambda,A)\| <
    \infty
  \end{math}
  by the Laplace transform representation of $R(\phdot,A)$. This readily
  implies that all poles of $R(\phdot,A)$ on $\sigma_{\per}(A)\subseteq
  i \bbR$ are of order $1$. Conversely, if $\sigma_{\per}(A)$ consists
  of first order poles of $R(\phdot,A)$, then by (ii) there are only
  finitely many of them, say $i\beta_1,...,i\beta_n\in i\bbR$, and the 
  spectral space associated to any of them consists of eigenvectors. Hence the
  spectral space $E_0$ associated with $\sigma_{\per}(A)$ is given by
  $\oplus_{k=1}^n \ker(i\beta_kI - A)$ and $(e^{tA})_{t\geq 0}$ is therefore
  bounded on $E_0$. We already know from (ii) that $(e^{tA})_{t\geq 0}$
  is uniformly exponentially stable and thus bounded on the spectral
  space associated with $\sigma(A) \setminus \sigma_{\per}(A)$. Hence,
  it is bounded on $E=E_0\oplus E_1$.
\end{remark}

\begin{proof}[Proof of Theorem~\ref{thm:semigroups-asymptotic}]
  We may assume that $\spb(A) = 0$.
 
  ``(i) $\Rightarrow$ (ii)'' Since the semigroup is bounded and since
  $\spb(A)$ is a dominant spectral value, it follows from \cite[Theorem~2.4]{MR933321}
  that the semigroup converges strongly to $0$ on $\ker P$. Moreover, 
  $0$ must be a simple pole of $R(\phdot,A)$ due to the boundedness of the semigroup.
  Hence, $e^{tA}|_{\im P} = \id_{\im P}$ for all $t \ge 0$. We thus
  conclude that for every $f\geq 0$
  \begin{displaymath}
    \distPos{e^{tA}f}
    \leq \distPos{Pf} + \|e^{tA}(f-Pf)\|
    =\|e^{tA}(f-Pf)\| \to 0
  \end{displaymath}
  as $t\to\infty$. Hence, the semigroup is individually asymptotically
  positive.
 
  ``(ii) $\Rightarrow$ (i)'' Let $P_{\per}$ be the spectral projection
  associated with $\sigma_{\per}(A)$. Note that $\sigma_{\per}(A)$
  consists of simple poles of $R(\phdot,A)$ since the semigroup is
  bounded.  Hence, by virtue of \cite[Proposition~2.3]{DanersI}, we can
  find a sequence $0 \le t_n \to \infty$ such that $e^{t_nA}P_{\per}f
  \to P_{\per}f$ for all $f \in E$.  For every $f \in E_+$ and every $t
  \ge 0$ we have
  \begin{displaymath}
    \distPos{e^{tA}P_{\per}f}
    = \lim_{n \to \infty} \distPos{e^{(t+t_n)A}P_{\per}f}
    \leq \lim_{n \to\infty} \distPos{e^{(t+t_n)A}f}
    + \lim_{n \to \infty}\|e^{(t+t_n)A}(P_{\per}f-f)\| = 0 \text{,}
  \end{displaymath}
  where the last limit is $0$ because the semigroup converges strongly
  to $0$ on $\ker P_{\per}$ (this follows from
  \cite[Theorem~2.4]{MR933321}). Hence $e^{tA}P_{\per} \ge 0$ for all $t
  \ge 0$. In particular $P_{\per} =e^{0A}P_{\per}\ge 0$.  Thus, $\im
  P_{\per}$ is a complex Banach lattice with respect to an appropriate
  equivalent norm, see
  \cite[Proposition~III.11.5]{Schaefer1974}. Moreover, we have shown
  that the $C_0$-semigroup $(e^{tA}|_{\im P_{\per}})_{t \ge 0}$ is
  positive. As $\sigma_{\per}(A|_{\im P_{\per}})\neq\emptyset$ we
  conclude that $\spb(A) = \spb(A|_{\im P_{\per}}) \in \sigma(A|_{\im
    P_{\per}}) \subseteq \sigma(A)$. Moreover, $\sigma_{\per}(A|_{\im
    P_{\per}}) = \sigma_{\per}(A)$ is imaginary additively cyclic; see
  \cite[Definition~B-III.2.5, Proposition~C-III.2.9 and
  Theorem~C-III.2.10]{Arendt1986}.  Since $\sigma_{\per}(A)$ is finite,
  it follows that $\sigma_{\per}(A)=\{0\}$.  This in turn implies that
  $P = P_{\per} \ge 0$.
 
  ``(i) $\Rightarrow$ (iii)'' If (i) is true, then according to
  \cite[Theorem~2.4]{MR933321}, $e^{tA}\to 0$ strongly on $\ker P$ as $t
  \to \infty$. Since $e^{tA}|_{\im P} = \id_{\im P}$ we have 
  $e^{tA}\to P\geq 0$ strongly. In particular (iii)
  holds with $Q = P$.
 
  ``(iii) $\Rightarrow$ (i)'' If $\spb(A)$ is not a spectral value or
  not dominant, then there is an eigenvalue $\lambda \in i\bbR
  \setminus\{0\}$. Hence, $e^{tA}$ does not converge strongly as $t \to
  \infty$. Thus (iii) implies that $\spb(A)$ must be a dominant spectral
  value. Then, however, $e^{tA}$ converges strongly to $P$ as $t \to
  \infty$, which in turn implies $P = Q \ge 0$.
 
  Clearly, (iv) implies (ii). Now, assume that (i) is true and that 
  $(e^{tA})_{t \ge 0}$ is uniformly exponentially stable on $\ker P$.
  If $\varepsilon > 0$, then we can find a $t_0 \ge 0$ such that
  $\|e^{tA}|_{\ker P}\| \le \varepsilon$ for all $t \ge t_0$. This
  implies that
  \begin{displaymath}
    \distPos{e^{tA}f} 
    \leq \distPos{Pf} + \|e^{tA}(f-Pf)\|
    \leq \varepsilon \|f-Pf\|
    \leq \varepsilon \|\id - P\| \|f\|
  \end{displaymath}
  for all $f \in E_+$ and all $t \ge t_0$ and so $(e^{tA})_{t \ge 0}$ is
  uniformly asymptotically positive.
\end{proof}

We shall now give several (counter-) examples regarding the assumptions
and conditions in Theorem~\ref{thm:semigroups-asymptotic}.

\begin{examples}
  \label{examples:semigroups-asymptotic}
  (a) The assumption that $\sigma_{\per}(A)$ be non-empty is essential
  in Theorem~\ref{thm:semigroups-asymptotic}. Indeed, let $B \in
  \bbR^{2\times 2}$ be a matrix with $\sigma(B) = \{-i,i\}$ and define
  $A_n = nB - \frac{1}{n}$ for every $n\in\bbN$. If we endow $\bbC^2$
  with the Euclidean norm and let $E = l^2(\bbN;\bbC^2)$, then $E$ is a
  complex Banach lattice. Let $A$ be the matrix multiplication operator
  on $E$ with symbol $(A_n)_{n \in \bbN}$, that is,
  \begin{equation}
    \label{eq:matrix-mult-op}
    D(A):= \{(x_n) \in E\colon (A_nx_n) \in E\},\qquad
    A(x_n):=(A_nx_n) \text{.}
  \end{equation}
  Then $A$ generates a bounded $C_0$-semigroup on $E$ and $\sigma(A) =
  \{\pm ni - \frac{1}{n}\}$. In particular,
  $\sigma_{\per}(A)=\emptyset$, which implies that $e^{tA}\to 0$
  strongly as $t \to \infty$ (see \cite[Theorem~2.4]{MR933321}). Since
  $\spb(A) = 0$, this implies that $(e^{tA})_{t \ge 0}$ is in particular
  trivially individually asymptotically positive, even though $\spb(A)$
  is not a spectral value of $A$.
 
  (b) If the semigroup $(e^{tA})_{t \ge 0}$ is not uniformly
  exponentially stable on the spectral space associated with $\sigma(A)
  \setminus \sigma_{\per}(A)$ then assertions (i)--(iii) of
  Theorem~\ref{thm:semigroups-asymptotic} do not imply (iv) in
  general. To see this, let $B \in \bbR^{3 \times 3}$ be such that
  $e^{tB}$ is the rotation of angle $t$ about the line in the direction
  of the vector $(1,1,1)$. Let $Q$ be the projection along $(1,1,1)$
  onto its orthogonal complement and define $A_n = nB - \frac{1}{n}Q$
  for every $n\in\bbN$.  Endow $\bbC^3$ with the Euclidean norm and
  consider the complex Banach lattice $E = l^2(\bbN; \bbC^3)$. If $A$ is
  the matrix multiplication operator on $E$ with symbol $(A_n)$
  analogous to \eqref{eq:matrix-mult-op}, then $A$ generates a bounded
  $C_0$-semigroup on $E$ which is individually but not uniformly
  asymptotically positive. Moreover,
  \begin{displaymath}
    \sigma(A)
    = \{0\}\cup\Bigl\{\pm ni-\frac{1}{n}\colon n\in\mathbb N \Bigr\}\text{,}
  \end{displaymath}
  so that $\sigma_{\per}(A) = \{0\}$, where $0$ is a simple pole of the
  resolvent. Thus, all assumptions of
  Theorem~\ref{thm:semigroups-asymptotic} are fulfilled, but assertions
  (i) and (iv) are not equivalent.
 
  (c) The assumption in Theorem~\ref{thm:semigroups-asymptotic}
  that the peripheral spectrum consist of finitely many poles of the
  resolvent cannot simply be omitted. Indeed, let $B$
  and $Q$ be as in (b), but this time, define $A_n = B -
  \frac{1}{n}Q$. As above, denote by $A$ the matrix multiplication
  operator with symbol $(A_n)$ on $E = l^2(\bbN;\bbC^3)$. Then $A$
  generates a bounded $C_0$-semigroup which is easily seen to be
  individually asymptotically positive 
  (though not uniformly asymptotically
  positive). However, $\sigma(A) = \{0,\pm i\} \cup \{\pm i-1/n\colon n
  \in \bbN\}$, so $\spb(A) = 0$ is not a dominant spectral value.
\end{examples}

\begin{remark}
  \label{rem:unif-asymp-resolvent-vs-ind-asymp-sg}
  In Example~\ref{examples:semigroups-asymptotic}(b) the semigroup
  $(e^{tA})_{t \ge 0}$ is not uniformly asymptotically positive.
  However, according to Theorem~\ref{thm:semigroups-asymptotic} the
  spectral projection $P$ associated with $0$ is positive.  Because the
  assumptions of Theorem~\ref{thm:resolvents-asymptotic} are fulfilled
  the \emph{resolvent} $R(\phdot,A)$ is uniformly asymptotically
  positive at $\spb(A)$. In particular, the resolvent of a generator can
  be uniformly asymptotically positive at $\spb(A)$ even if the
  semigroup is only individually asymptotically positive.
\end{remark}

\section{Applications of asymptotic positivity}
\label{section:applications-asymptotic}

In this penultimate section we shall give some applications of our
results on asymptotic positivity. We begin with an analysis of the
finite-dimensional case. Then we revisit the bi-Laplacian with Dirichlet
boundary conditions and formulate a result on asymptotic positivity of
the resolvent which in some manner complements
Proposition~\ref{prop:resolvent-bi-laplacian}. We again consider the
special case of self-adjoint operators on Hilbert spaces, with an
application to the Dirichlet-to-Neumann operator on
$L^2(\partial\Omega)$, as well as a transport process on a metric graph
and a one-dimensional delay differential equation.

\paragraph*{The finite-dimensional case}

We consider the special case of matrices $A \in \bbC^{n \times n}$ and
characterise when the matrix exponential $(e^{tA})_{t \ge 0}$ is
asymptotically positive.

A characterisation of eventual strong positivity of matrix semigroups
was first given in \cite[Theorem~3.3]{Noutsos2008}, and later in
\cite[Theorem~6.1]{DanersI} as an application of the general
$C(K)$-theory.  By characterising \emph{asymptotically} positive matrix
semigroups, Theorem~\ref{thm:asymp-finite-dimensional} below adds
new aspects to the finite-dimensional theory. Moreover, since the matrix
$A$ in Theorem~\ref{thm:asymp-finite-dimensional} is not required to be
real, the theorem also contributes to the Perron--Frobenius theory of
matrices with entries in $\mathbb C$, a topic which was the focus of
\cite{Noutsos2012}. We also refer to \cite{Rump2003}, where
generalisations of Perron--Frobenius theory to complex matrices are
approached from a rather different perspective.

It is evident that in finite dimensions individual and uniform
asymptotic positivity are equivalent. Hence we merely speak of ``asymptotic
positivity''.

\begin{theorem}
  \label{thm:asymp-finite-dimensional}
  Let $A \in \bbC^{n \times n}$ and assume that $(e^{t(A-\spb(A))})_{t
    \ge 0}$ is bounded or equivalently that all $\lambda \in
  \sigma_{\per}(A)$ are simple poles of $R(\phdot,A)$. Then the
  following assertions are equivalent:
  \begin{enumerate}[\upshape (i)]
  \item $(e^{tA})_{t \ge 0}$ is asymptotically positive.
  \item There is a number $c \in \bbR$ such that $A+c\id$ has positive
    spectral radius, $\frac{A+c\id}{\spr(A+c\id)}$ is power bounded and
    $A+c\id$ is asymptotically positive.
  \end{enumerate}
  \begin{proof}
    ``(i) $\Rightarrow$ (ii)'' It follows from
    Theorem~\ref{thm:semigroups-asymptotic} that $\spb(A)$ is a dominant
    spectral value and that the associated spectral projection $P$ is
    positive. Now, choose $c > 0$ sufficiently large such that
    $\spb(A)+c\id > 0$ is larger than the modulus of any other spectral
    value of $A+c\id$. Then in particular $r := \spr(A+c\id) = \spb(A) +
    c$. The spectral projection of $\frac{A+c\id}{r}$ associated with
    the spectral value $1$ is $P$, and $\im P$ coincides with the fixed
    space of $\frac{A+c\id}{r}$. Thus, we clearly have
    \begin{displaymath}
      \Bigl(\frac{A+c\id}{r}\Bigr)^n \to P \ge 0
      \quad \text{as $n\to\infty$,}
    \end{displaymath}
    which implies that $\frac{A+c\id}{r}$ is power bounded and that
    $A+c\id$ is asymptotically positive.
 
    ``(ii) $\Rightarrow$ (i)'' First, let $T \in \bbC^{n \times n}$,
    $\spr(T) > 0$, such that $\frac{T}{\spr(T)}$ is power bounded and
    such that $T$ is asymptotically positive. Let $Q$ be the spectral
    projection associated with $\sigma(T) \cap \spr(T)\bbT$, where
    $\bbT$ denotes the unit circle in $\bbC$. Since $\frac{T}{\spr(T)}$
    is power bounded, all eigenvalues on the circle $\spr(T)\bbT$ are
    simple poles of the resolvent, so the image of $Q$ is spanned by
    eigenvectors of $T$ belonging to eigenvalues of modulus $\spr(T)$.
 
    We can find a sequence $n_k \to \infty$ of positive integers such
    that $\bigl(\lambda\spr(T)^{-1}\bigr)^{n_k} \to 1$ as $k \to \infty$
    for each $\lambda \in \spr(T) \bbT$; this follows from the same
    argument that was used in the proof of
    \cite[Proposition~2.3]{DanersI}. This implies that
    \begin{displaymath}
      \Bigl(\frac{T}{\spr(T)}\Bigr)^{n_k} \to Q
      \quad \text{as $k \to \infty$,}
    \end{displaymath}
    which in turn shows that $Q \ge 0$. Hence, $\im Q$ is a
    (finite-dimensional) complex Banach lattice when equipped with an
    appropriate norm, see
    \cite[Proposition~III.11.5]{Schaefer1974}. Moreover, for each $0 \le
    f \in \im Q$, we have
    \begin{displaymath}
      \frac{T}{\spr(T)}f
      = \lim_{k \to \infty}\Bigl(\frac{T}{\spr(T)}\Bigr)^{n_k+1}f \ge 0 \text{,}
    \end{displaymath}
    so $T|_{\im Q} \ge 0$. This implies that the spectral radius
    $\spr(T|_{\im Q}) = \spr(T)$ is contained in $\sigma(T|_{\im Q})$
    and hence in $\sigma(T)$; see \cite[Proposition~4.1.1(i)]{Meyer-Nieberg1991}.
 
    Next we show that the spectral projection $P$ associated with
    $\spr(T)$ is positive.  Using the Neumann series expansion
    $R(\lambda,T)=\sum_{k=0}^\infty \lambda^{-{k+1}}T^k$ valid for
    $|\lambda|>\spr(T)$ we have
    \begin{displaymath}
      \begin{split}
        \distPos{Pf}
        &= \distPos{\lim_{\lambda \downarrow \spr(T)}
          (\lambda - \spr(T))R(\lambda,T)f} \leq
        \limsup_{\lambda\downarrow\spr(T)} \Bigl((\lambda - \spr(T))
        \sum_{n=0}^\infty
        \frac{\distPos{T^n f}}{\lambda^{n+1}} \Bigr) \\
        & = \limsup_{\lambda \downarrow \spr(T)}
        \Bigl(\frac{\lambda}{\spr(T)} - 1\Bigr)\sum_{n=0}^\infty
        \frac{\distPos{(\frac{T}{\spr(T)})^nf}}
        {\Bigl(\frac{\lambda}{\spr(T)}\Bigr)^{n+1}} = 0
      \end{split}
    \end{displaymath}
    for all $f \ge 0$, where we have used that
    $\distPos{\bigl(\frac{T}{\spr(T)}\bigr)^nf}\to 0$ as $n \to
    \infty$. Hence, $P \ge 0$.
    
    Finally assume that $A+c\id$ fulfils condition (ii). Then by what we
    have just shown $\spr(A+c\id)\in\sigma(A+c\id)$, and hence
    $\spr(A+c\id)=\spb(A+c\id)$ is a dominant spectral
    value of $A+c\id$. Moreover, the associated spectral projection $P$ is
    positive. Hence, $\spb(A)$ is a dominant spectral value of $A$, and
    since the associated spectral projection is still $P$, we conclude
    from Theorem~\ref{thm:semigroups-asymptotic} that $(e^{tA})_{t \ge
      0}$ is asymptotically positive.
  \end{proof}
\end{theorem}

In \cite[Proposition~6.2]{DanersI} we proved that real, eventually
positive semigroups in two dimensions are automatically positive.  This
is also true for asymptotically positive semigroups.

\begin{proposition}
  \label{prop:asymp-two-dimensional}
  Let $A \in \bbR^{2 \times 2}$ such that $(e^{t(A-\spb(A))})_{t \ge 0}$
  is bounded.  If the semigroup $(e^{tA})_{t \ge 0}$ is asymptotically
  positive, then it is positive.
  \begin{proof}
    We may assume $\spb(A) = 0$. From
    Theorem~\ref{thm:semigroups-asymptotic} we know that $\spb(A) = 0$
    is a dominant spectral value of $A$. The boundedness of
    $(e^{t(A-\spb(A))})_{t \ge 0}$ implies that $\spb(A)$ is a simple
    pole of the resolvent and hence the algebraic and geometric
    muliplicities coincide. If the multiplicity of $0$ is two, then $A =
    0$ and the semigroup is positive.
 
    Now assume that $A$ has two distinct simple eigenvalues. Then $A$
    has a real eigenvalue $\lambda<0=\spb(A)$. If $P$ is the spectral
    projection associated with $0$, then $P \ge 0$ by
    Theorem~\ref{thm:semigroups-asymptotic}, and hence for every $t \ge
    0$
    \begin{displaymath}
      e^{tA} = P + e^{\lambda t}(\id - P)
      = e^{\lambda t}\id + (1-e^{\lambda t})P
      \ge 0
    \end{displaymath}
    as $e^{\lambda t} \le 1$.
  \end{proof}
\end{proposition}

The reader should note that the assertion of the above proposition fails
if $A$ is allowed to be a complex matrix; for example the semigroup
generated by the matrix
\begin{displaymath}
  A =
  \begin{bmatrix}
    0 & 0 \\
    0 & -1+i
  \end{bmatrix}
\end{displaymath}
is asymptotically positive, but not positive.

\paragraph*{The resolvent of the bi-Laplace operator with Dirichlet
  boundary conditions}

Here we consider the same operator $A_p$ ($1 < p < \infty$) as in the
second paragraph of Section~\ref{section:applications-strong} and we use
the properties of $A_p$ given in
Proposition~\ref{prop:bi-laplacian-on-lp}. The eventual positivity of
the resolvent $R(\phdot,A_p)$ was analysed in
Proposition~\ref{prop:resolvent-bi-laplacian}. The disadvantage there is
that, in contrast to the semigroup, our results on eventual strong
positivity only apply to large $p$ and/or small dimensions $n$. We now
show that the resolvent is at least asymptotically positive for
\emph{all} $p \in (1,\infty)$, independent of the dimension.

\begin{proposition}
  \label{prop:resolvent-bi-laplacian-asymptotic}
  Let $p \in (1,\infty)$ and let $\Omega \in C^\infty$ be such that the
  conclusion of Theorem~\ref{thm:positive-eigenfunction-bi-laplacian}
  holds. Then the resolvent
  $R(\phdot,A_p)$ is uniformly asymptotically positive at $\spb(A_p)$.
  \begin{proof}
    According to Lemma~\ref{lemma:spectral-bound-bi-laplacian}
    $\spb(A_p)$ is a simple pole of the resolvent $R(\phdot,A_p)$ and a
    dominant spectral value of $A_p$; moreover, the corresponding
    spectral projection is positive. The assertion therefore follows
    from Theorem~\ref{thm:resolvents-asymptotic}.
  \end{proof}
\end{proposition}

\paragraph*{Asymptotic positivity for self-adjoint operators on Hilbert
  lattices and the Dirichlet-to-Neumann operator}

In this section we again consider self-adjoint operators $A$ on a
Hilbert lattice, c.f.~the corresponding paragraph in
Section~\ref{section:applications-strong}.  In
Theorem~\ref{thm:evtl-pos-hilbert-lattice} we provided a
characterisation of eventual strong positivity under the assumption that
$D(A) \subseteq E_u$ for some $u \gg 0$. If we do not assume this
domination property, we are still able give a sufficient condition
for the asymptotic positivity of the resolvent and the semigroup.

\begin{theorem}
  \label{th:hilbert-asymptotic}
  Let $H$ be a complex Hilbert lattice and let $A$ be a densely defined,
  self-adjoint operator on $H$ such that $\spb(A)\in\bbR$ is an isolated
  point of the spectrum of $A$. Moreover, assume that the eigenspace
  $\ker(\spb(A)\id-A)$ is one-dimensional and contains a non-zero
  positive vector. Then the resolvent $R(\phdot,A)$ is uniformly
  asymptotically positive at $\spb(A)$ and the semigroup $(e^{tA})_{t
    \ge 0}$ is uniformly asymptotically positive.
\end{theorem}
\begin{proof}
  Using the same argument as at the beginning of the proof of 
  Theorem~\ref{thm:evtl-pos-hilbert-lattice} we
  deduce that the Banach space adjoint of $A$ has a positive functional
  as an eigenvector for the eigenvalue $\spb(A)$. Since $A$ is
  self-adjoint, $\spb(A)$ is a simple pole of the resolvent and
  therefore Corollary~\ref{cor:projections-positive} implies that the
  spectral projection $P$ associated with $\spb(A)$ is positive. Hence,
  $R(\phdot,A)$ is uniformly asymptotically positive at $\spb(A)$ by
  Theorem~\ref{thm:resolvents-asymptotic}. Because $\spb(A)$ is a
  dominant spectral value and since $(e^{tA)})_{t \ge 0}$ is analytic,
  the semigroup $(e^{t(A-\spb(A)I)})_{t \ge 0}$ is bounded; since by
  assumption $\spb(A)$ is isolated (and $\sigma(A) \subseteq \bbR$),
  $(e^{t(A-\spb(A)I)})_{t \ge 0}$ is even uniformly exponentially stable
  on $\ker P$. Hence, $(e^{tA)})_{t \ge 0}$ is uniformly asymptotically
  positive by Theorem~\ref{thm:semigroups-asymptotic}.
\end{proof}

Of course, similar assertions as in the above theorem also hold for the
eventual positivity of the resolvent at other spectral values than
$\spb(A)$.

\begin{remarks}
  (a) Suppose $A$ is a densely defined, self-adjoint operator on a
  Hilbert lattice such that $\spb(A)$ is isolated in $\sigma(A)$. Then
  Theorem~\ref{thm:semigroups-asymptotic} implies that a
  \emph{necessary} condition for the asymptotic positivity (uniform or
  individual) of $(e^{tA})_{t \ge 0}$ is that the spectral projection $P$
  associated with $\spb(A)$ is positive. In this case, we can also conclude
  the existence of a positive eigenvector for the eigenvalue $\spb(A)$
  since $\spb(A)$ is automatically a simple pole of the resolvent.

  (b) One might wonder whether under the assumptions of the above
  theorem, the semigroup $(e^{tA})_{t \ge 0}$ is individually
  \emph{eventually} positive. If we set $p=2$ in
  Example~\ref{example:domination-property}, we can see that this is not
  in general the case without a domination or smoothing assumption.
\end{remarks}

\begin{example}
  We recall the Dirichlet-to-Neumann operator $D_\lambda$ from
  Section~\ref{section:applications-strong}. Here we may assume that
  $\Omega \subseteq \bbR^n$ is a general bounded domain with sufficiently
  smooth boundary; as mentioned earlier, $D_\lambda$ is a densely
  defined, self-adjoint operator on $L^2(\partial\Omega)$ with compact
  resolvent. It follows directly from the definition that the eigenspace
  associated with $\spb(-D_\lambda)$ is given by the finite-dimensional
  span in $L^2(\partial\Omega)$ of the traces of all eigenfunctions of
  the Laplacian associated with the Robin problem
  \begin{equation}
    \label{eq:robin-dtn}
    \Delta f = \lambda f \quad\text{in $\Omega$}, \qquad 
    \frac{\partial}{\partial\nu}f = \spb(-D_\lambda)f \quad 
    \text{on $\partial\Omega$}.
  \end{equation}
  By Theorem~\ref{th:hilbert-asymptotic} we have that
  $(e^{-tD_\lambda})_{t\ge 0}$ is uniformly asymptotically positive if
  there is a solution of \eqref{eq:robin-dtn} which is unique up to
  scalar multiples and which has non-zero positive trace on
  $\partial\Omega$. Conversely, a necessary condition for the asymptotic
  positivity of $(e^{-tD_\lambda})_{t\ge 0}$ is the \emph{existence} of
  (at least one) solution of \eqref{eq:robin-dtn} with positive trace.
\end{example}

\begin{example}
  Consider the example of Bose condensation from
  Theorem~\ref{thm:bose-condensation}, but without the assumption that
  $q$ be real valued. The generator $-A$ is no longer real, but from
  Theorem~\ref{th:hilbert-asymptotic} we still conclude that the
  corresponding semigroup is uniformly asymptotically positive.
\end{example}

\paragraph*{A network flow with non-positive mass diversion}
Consider a directed graph with $n$ edges $e_k$ of length $l_k$,
$k=1,\ldots,n$, and suppose that we are given a mass distribution on
every edge. Further, assume that a transport process shifts the mass
along the edges with a given velocity. Whenever some mass arrives at a
vertex, it is diverted to the outgoing edges of this vertex according to
some pre-defined weights. Such a transport process is often called a
\emph{network flow} and it can be described by means of a
$C_0$-semigroup on the space $\bigoplus_{k=1}^n L^1([0,l_k])$.

During the last decade a deep and extensive theory of network flow
semigroups has been developed which deals, among other topics, with the
long time behaviour of the flow and relates it to properties of the
underlying graph; see e.g.~\cite{Kramar2005, Dorn2009, Dorn2010}.
However, it seems that so far only positive weights for the mass
diversion in the vertices have been considered. In this section we want
to demonstrate by means of an example that it is possible to consider
non-positive mass diversion and that in such a situation asymptotic
positivity can occur. It is however not our intention to develop a
general theory here.

We consider a directed graph as shown in \eqref{eq:graph}. It consists
of two vertices $v_{-1},v_0 $, an edge $e_1$ of length $1$ directed from
$v_{-1}$ to $v_0$, a ``looping'' edge $e_2$ of length $1$ going from
$v_0$ to itself and another ``looping'' edge $e_3$ of lenght $l$, again
going from $v_0$ to itself. We assume $l>0$ to be an irrational number.
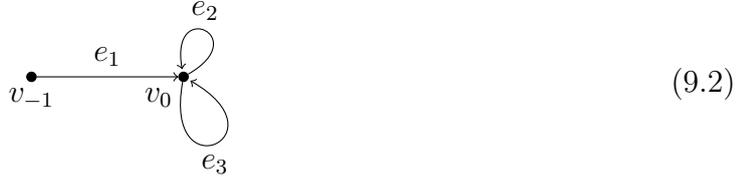
\begin{equation}
  \label{eq:graph}
  \begin{tikzpicture}[scale=2,baseline=-.2cm,
    nw node/.style={circle,draw,fill=black,inner sep=1.25pt}]
    \useasboundingbox  (-.2,-.6) rectangle (1.5,.5);
    \node[nw node] (v0) at (1,0) {};
    \node[nw node] (v1) at (0,0) {};
    \path (v0) node[below left] {$v_0$};
    \path (v1) node[below] {$v_{-1}$};
    \draw (v1) edge[->]  node[above] {$e_1$} (v0);
    \draw (v0) edge[in=100,out=30,loop] node[above] {$e_2$} (); 
    \draw[scale=1.5] (v0) edge[in=-30,out=-100,loop] node[below] {$e_3$} (); 
  \end{tikzpicture}
\end{equation}
\par
We assume that the mass is shifted along the edges with constant
velocity $1$, and that the mass diversion in the vertices is as follows:
Since $v_{-1}$ has no incoming edges, no mass arrives at $v_{-1}$ and
hence no mass is inserted to $e_1$ from $v_{-1}$. Two thirds of the mass
arriving at $v_0$ from $e_1$ is diverted to $e_3$; the other third is
diverted to $e_2$, but with a flipped sign. One half of the mass
arriving at $v_0$ from $e_2$ is diverted to $e_2$ itself, the other half
is diverted to $e_3$.  Similarly, one half of the mass arriving at $v_0$
from $e_3$ is diverted to $e_2$ and the other half is diverted to $e_3$
itself.

Note that the mass diversion in $v_0$ contains a somewhat finer
structure than is usually considered in the literature: in most other
models, all the incoming mass at a vertex is summed up, and then the
entire mass is distributed to the outgoing edges according to certain
weights. In our model, however, the diversion of mass in $v_0$ depends
on the edge it arrives from.

We model the mass distribution by a function $f = (f_1,f_2,f_3) \in
L^1([0,1]) \oplus L^1([0,1]) \oplus L^1([0,l]) =: E$, where the
space $E$ is endowed with the norm 
$\|(f_1,f_2,f_3)\| = \|f_1\|_1 + \|f_2\|_1 + \|f_3\|_1$. Here, $f_k$
describes the mass distribution on the edge $e_k$; for each $k=1,2,3$
the number $0$ in the interval $[0,1]$ (or $[0,l]$, respectively) shall
denote the \emph{starting point} of the edge $e_k$. The time evolution
of our network flow can be described by the abstract Cauchy problem
$df/dt = Af$ where the operator $A$ on $E$ is given by
\begin{displaymath}
  \begin{aligned}
    D(A) &=  \Bigl\{(f_1,f_2,f_3)
    \in W^{1,1}((0,1)) \oplus W^{1,1}((0,1)) \oplus W^{1,1}((0,l))\colon\\
    &\hspace{5em} f_2(0) 
    = \frac{1}{2}f_2(1) + \frac{1}{2}f_3(1) - \frac{1}{3}f_1(1), \\
    &\hspace{5em} f_3(0) 
    = \frac{1}{2}f_2(1) + \frac{1}{2}f_3(1) + \frac{2}{3}f_1(1),
    \quad f_1(0) = 0 \Bigr\}, \\
    A(f_1,f_2,f_3) &= -(f_1',f_2',f_3').
  \end{aligned}
\end{displaymath}
Now we can prove the following properties of the abstract Cauchy problem
associated with $A$:

\begin{theorem}
  The operator $A$ defined above is closed, densely defined and has the
  following properties:
  \begin{enumerate}[\upshape (i)]
  \item Any complex number $\lambda$ is an eigenvalue of $A$ if and only
    if it is a spectral value of $A$ if and only if the matrix
    \begin{displaymath}
      S(\lambda) :=
      \begin{bmatrix}
        e^{-\lambda} - 2 & e^{-\lambda l} \\
        e^{-\lambda} & e^{-\lambda l} - 2
      \end{bmatrix}
    \end{displaymath}
    is singular. Moreover, for each eigenvalue $\lambda$ of $A$ the
    corresponding eigenspace is one-dimensional.
  \item $A$ is dissipative and generates a contractive $C_0$-semigroup
    on $E$.
  \item $A$ has compact resolvent.
  \item$(0,\one_{[0,1]}, \one_{[0,l]}) \in \ker(A)$ and, moreover,
    $(\frac{1}{3}\one_{[0,1]},\one_{[0,1]}, \one_{[0,l]}) \in \ker(A')$.
  \item $\spb(A)$ equals $0$ and is a dominant spectral value of $A$.
  \item The semigroup $(e^{tA})_{t \ge 0}$ is individually
    asymptotically positive, but not positive.
  \end{enumerate}
\end{theorem}
\begin{proof}
  Obviously, $A$ is closed and densely defined.
	
  (i) Let $\lambda$ be a complex number. A straightforward computation
  shows that there exists a non-trivial function $f \in \ker(\lambda\id
  - A)$ if and only if $S(\lambda)$ is singular. Moreover, if such a
  function $f$ exists, then the same computation shows that $f$ is
  unique up to scalar multiples, so $\ker(\lambda\id-A)$ is
  one-dimensional. Finally, another simple (but somewhat lengthy)
  computation shows that $\lambda\id- A$ is surjective if $S(\lambda)$
  is not singular. This proves (i).
	
  (ii) Since $\det S(\lambda)$ is an entire function which is not identically $0$,
  $S(\lambda)$ must be regular for some $\lambda > 0$; hence $\lambda
  \in \rho(A)$ for some $\lambda > 0$. Using the boundary 
  condition satisfied by functions in $D(A)$ it
  is easy to check that the restriction of $A$ to the real part $E_\bbR$ 
  of $E$ is indeed dissipative. Since $A$ is a real operator, the restriction 
  $A|_{E_\bbR}$ generates a contractive $C_0$-semigroup on the real Banach 
  space $E|_{E_\bbR}$. Thus, $A$ generates a $C_0$-semigroup on $E$, and it 
  follows from \cite[Proposition~2.1.1]{Fendler1998} that this semigroup
  is contractive, too. In particular, $A$ is dissipative.
	
  (iii) We have seen in (ii) (or we can conclude immediately from (i))
  that $\rho(A) \not= \emptyset$. Since $D(A)$ compactly embeds into
  $E$, it follows that the resolvent is compact.
	
  (iv) The first assertion is obvious and the second assertion can
  easily be checked by using the definition of the adjoint.
	
  (v) Since $A$ is dissipative and has non-empty resolvent set, 
  no spectral value of $A$ can have
  strictly positive real part (alternatively, we could also conclude
  this from (i)).  Since $0 \in \sigma(A)$ according to (iv), we have
  indeed $\spb(A) = 0$. Now assume for a contradiction that $i\beta$ ($0
  \not= \beta \in \bbR$) is another spectral value of $A$ on the
  imaginary axis. Then it follows from (ii) that
  \begin{displaymath}
    0 = \det S(i\beta) = -2(e^{-i\beta} + e^{-i\beta l}) + 4.
  \end{displaymath}
  Taking the real part of the above equation we obtain
  \begin{math}
    \cos (\beta) + \cos (\beta l) = 2
  \end{math}
  and hence $\cos(\beta)=\cos(\beta l) = 1$. Thus, $\beta \in \pi
  (\frac{1}{2} + \bbZ)$ and $\beta l \in \pi (\frac{1}{2} + \bbZ)$ which
  is a contradiction since $l$ is irrational.
	
  (vi) Due to (iii) $\spb(A) = 0$ is a pole of the resolvent and since
  the semigroup $(e^{tA})_{t \ge 0}$ is bounded, $0$ is even a first
  order pole of the resolvent. Since all eigenspaces of $A$ are one
  dimensional,
  Corollary~\ref{cor:projections-positive} together with assertion (iv)
  implies that the associated spectral projection $P$ is positive. As
  $\spb(A)$ is a dominant spectral value according to (v), individual
  asymptotic positivity of the semigroup follows from
  Theorem~\ref{thm:semigroups-asymptotic}.
	
  That the semigroup is not positive is obvious if we consider a
  positive initial mass distribution $f$ which lives only on the first
  edge: after some time some of the mass of $f$ is diverted with a
  negative sign to the second edge, and when this first happens, there
  is no mass close the end of $e_2$ and $e_3$ which could compensate
  those negative values.
\end{proof}

Note in the proof of (i) that we cannot replace the second computation
with a compactness argument: equality of the point spectrum and the
spectrum does \emph{not} follow from the compactness of the embedding
$D(A) \hookrightarrow E$ as long as we have not shown that $\sigma(A)
\neq \bbC$, and to show this, we need assertion (i).

\paragraph*{A delay differential equation}

In \cite[Section~6.5]{DanersI} a delay differential equation was
considered as an example for eventual strong positivity on a
$C(K)$-space.  Here, we want to consider another delay differential
equation whose solution semigroup is only asymptotically positive.

Consider a time-dependent complex value $y(t)$ whose time evolution is
governed by the equation
\begin{equation}
  \label{eq:delay-equation}
  \dot y(t) = y(t-2) - y(t-1).
\end{equation}
As shown in \cite[Section~VI.6]{Engel2000}, this equation can be
rewritten as an abstract Cauchy problem $\dot u = Au$ in the space $E =
C([-2,0])$, where the operator $A$ is given by
\begin{equation}
  \begin{aligned}
    \label{form:generator-of-dealy-semigroup}
    D(A) & = \{f \in C^1([-2,0])\colon f'(0) = f(-2) - f(-1)\} \\
    Af & = f'
  \end{aligned}
\end{equation}
(one has to set $r = 2$, $Y = \bbC$, $B = 0$ and $\Phi(f) = f(-2) -
f(-1)$ in \cite[Section~VI.6]{Engel2000} to obtain our example). 
There it is also shown that the operator
$A$ generates a $C_0$-semigroup on $E$. We are now going to prove the
following theorem on this semigroup.

\begin{theorem}
  \label{thm:delay-equation}
  Let the operator $A$ on $E = C([-2,0])$ be given by
  \eqref{form:generator-of-dealy-semigroup}. Then the operator $A$ and
  the $C_0$-semigroup $(e^{tA})_{t \ge 0}$ on $E$ have the following
  properties:
  \begin{enumerate}[\upshape (i)]
  \item $A$ has compact resolvent and the spectral bound $\spb(A)$
    equals $0$ and is a dominant spectral value.
  \item $\spb(A)$ is an algebraically simple eigenvalue of $A$ and the
    associated spectral projection $P$ is positive.
  \item $(e^{tA})_{t \ge 0}$ is uniformly asymptotically positive, but
    neither positive nor individually eventually strongly positive with
    respect to any quasi-interior point of $E_+$.
  \end{enumerate}
\end{theorem}
\begin{proof}
  (i) By the Arzel\`{a}--Ascoli Theorem the embedding $D(A)
  \hookrightarrow E$ is compact, and since $A$ has non-empty resolvent
  set its resolvent is compact. In particular, $\lambda\in\sigma(A)$ if
  and only if $\lambda$ is an eigenvalue of $A$.  A short computation
  shows that this is the case if and only if
  \begin{equation}
    \label{eq:characteristic-equation-for-dde}
    \lambda = e^{-2\lambda} - e^{-\lambda}.
  \end{equation}
  (alternatively, this follows from \cite[Proposition~VI.6.7]{Engel2000}).
  Obviously, $\lambda=0$ is a solution of
  \eqref{eq:characteristic-equation-for-dde}, so we have to show that
  \eqref{eq:characteristic-equation-for-dde} has no other solution with
  non-negative real part. If $\lambda \not= 0$ and if we set $z=\lambda/2$, then
  \eqref{eq:characteristic-equation-for-dde} is equivalent to
  \begin{equation}
    \label{eq:characteristic-equation-for-dde-1}
    e^{3z}=-\frac{\sinh z}{z}.
  \end{equation}
  It is easy to see that \eqref{eq:characteristic-equation-for-dde-1}
  does not have a solution on $i\mathbb R\setminus\{0\}$. We now show
  that \eqref{eq:characteristic-equation-for-dde-1} does not have a
  solution $z=\alpha+i\beta$ with $\alpha>0$ and $\beta\in\mathbb R$
  either. A short calculation using that $\sin^2\beta\leq\beta^2$ and
  $\sinh^2\alpha\geq\alpha^2$ shows that
  \begin{displaymath}
    \begin{split}
      \frac{|\sinh z|^2}{|z|^2}
      &=\frac{\sinh^2\alpha+\sin^2\beta}{\alpha^2+\beta^2}
      \leq \frac{\sinh^2\alpha+\beta^2}{\alpha^2+\beta^2}\\
      &=\frac{\sinh^2\alpha-\alpha^2}{\alpha^2+\beta^2}+1
      \leq\frac{\sinh^2\alpha-\alpha^2}{\alpha^2}+1
      =\frac{\sinh^2\alpha}{\alpha^2}.
    \end{split}
  \end{displaymath}
  Using the Taylor expansions for $\exp$ and $\sinh$ about $z=0$ we
  therefore have
  \begin{displaymath}
    |e^{3z}|
    =e^{3\alpha}
    >e^{\alpha}
    =\sum_{k=0}^\infty\frac{\alpha^k}{k!}
    >\sum_{k=0}^\infty\frac{\alpha^{2k}}{(2k)!}
    >\sum_{k=0}^\infty\frac{\alpha^{2k}}{(2k+1)!}
    =\frac{\sinh\alpha}{\alpha}
    \geq \frac{|\sinh z|}{|z|}
  \end{displaymath}
  for all $\alpha>0$. Hence \eqref{eq:characteristic-equation-for-dde-1}
  cannot have a solution with non-negative real part except for $z=0$.
  
  (ii) Since the resolvent of $A$ is compact, $\spb(A) = 0$ is a pole of
  the resolvent. To show that $0$ is an algebraically simple eigenvalue
  of $A$ we verify the assumptions of
  Proposition~\ref{prop:projection-positive-with-one-strictly-positive-vector}.
  To see that they are fulfilled, note that $\ker A$ is one-dimensional and spanned
  by the quasi-interior point $\one_{[-2,0]}$ of $E_+$.
  Moreover, one can easily check that the positive
  functional $\varphi \in E'$, given by $\varphi(f) = f(0) +
  \int_{-2}^{-1} f(x) \, dx$, is contained in $\ker A'$.
  
  (iii) Since the semigroup $(e^{tA})_{t \ge 0}$ is eventually
  norm-continuous \cite[Theorem~VI.6.6]{Engel2000} it follows from
  Theorem~\ref{thm:semigroups-asymptotic} and
  Remark~\ref{rem:eventually-norm-continuous-semigroup}
  that it is uniformly asymptotically positive. By
  \cite[Example~B-II.1.22]{Arendt1986} the semigroup is not positive.
  
  Finally, assume for a contradiction that the semigroup is individually
  eventually strongly positive with respect to a quasi-interior point
  $u$ of $E_+$. Since $u \gg_{\one_{[-2,0]}} 0$, the semigroup is then
  individually eventually strongly positive with respect to
  $\one_{[-2,0]}$ and Theorem~\ref{thm:semigroups-strong} implies that
  the spectral projection $P$ corresponding to $\spb(A) = 0$ fulfils $P
  \gg_{\one_{[-2,0]}} 0$. However,
  Proposition~\ref{prop:projections-strong} then yields that $\ker(A')$
  contains a strictly positive functional $\tilde \varphi$. Since $0$ is
  an algebraically simple eigenvalue of $A$, it is also an algebraically
  simple eigenvalue of $A'$ and hence $\ker A'$ is one-dimensional; see
  \cite[Section~III.6.6]{Kato1976}. Thus, $\tilde \varphi$ has to be a
  scalar multiple of the functional $\varphi$ from (b), which is clearly
  a contradiction, since $\varphi$ is not strictly positive.
\end{proof}

It is currently unclear whether the semigroup $(e^{tA})_{t \ge 0}$ is
\emph{individually eventually positive} in the sense that for each $f
\in E_+$ there exists $t_0 \ge 0$ such that $e^{tA}f \ge 0$ for all $t
\ge t_0$.

\section{Open problems}
\label{section:open-problems}

We have seen in several examples that the notion of ``eventual
positivity'' on a general Banach lattice is difficult from a structural
point of view, and therefore additional assumptions on the spectrum seem to be
necessary to obtain good descriptions. It is therefore natural to ask if
these assumptions can be changed or even weakened, and if there are
possible alternative definitions. Let us explicitly formulate the
following open problems:

(a) In our characterisations of strong eventual and asymptotic
positivity we always assumed the peripheral spectrum to be
finite. However, in some important examples, as e.g.~in some transport
equations, this assumption is not fulfilled. We therefore ask:

\begin{quote}
  \slshape How can asymptotic positivity of a semigroup be characterised
  if the peripheral spectrum $\sigma_{\per}(A)$ is allowed to be
  infinite and even unbounded?
\end{quote}

(b) Example~\ref{example:domination-property} shows that strong eventual
positivity of the resolvent or the semigroup cannot be characterised by
strong positivity of the spectral projection if the assumption $D(A)
\subseteq E_u$ is dropped. One could ask the following question:

\begin{quote}
  \slshape Suppose that all assumptions of
  Theorem~\ref{thm:semigroups-strong} are fulfilled except for the
  condition $D(A) \subseteq E_u$. Can individual eventual strong
  positivity of the semigroup still be characterised by individual
  eventual strong positivity of the resolvent at $\spb(A)$ plus a
  spectral condition?
\end{quote}

(c) We only defined the notion of asymptotic positivity of a semigroup
$(e^{tA})_{t \ge 0}$ under the assumption that the rescaled semigroup
$(e^{t(A-\spb(A))})_{t \ge 0}$ be bounded. If this assumption is not
fulfilled, it is not clear to the authors if the condition
$\distPos{e^{t(A-\spb(A))}f} \to 0$ for each $f \ge 0$ should still be
used to define individual asymptotic positivity, or if for instance the
condition $\frac{\distPos{e^{t(A-\spb(A))}f}}{\|e^{t(A-\spb(A))}\|} \to
0$ for each $f \ge 0$ would be more appropriate. The same question
arises for asymptotic positivity of the resolvent if the
Abel-boundedness condition in Definition~\ref{def:resolvents-asymptotic}
is dropped:

\begin{quote}
  \slshape How should asymptotic positivity of semigroups and resolvents
  be defined without additional boundedness assumptions?
\end{quote}

(d) The following problem is concerned with eventual positivity rather
than eventual strong positivity: In
Example~\ref{examples:ind-evtl-pos-is-complicated}(b) and in 
\cite[Example~8.2]{DanersI} we showed that
individual eventual positivity of a semigroup does not imply individual
eventual positivity of the resolvent at $\spb(A)$, even in finite
dimensions.  However, if the spectral bound $\spb(A)$ is a dominant
spectral value, one might ask whether at least the converse implication
is true:

\begin{quote}
  \slshape Let $(e^{tA})_{t \ge 0}$ be a $C_0$-semigroup with dominant
  spectral value $\spb(A)$ of $A$ and suppose that the resolvent is
  individually eventually positive at $\spb(A)$.  Does it
  follow (maybe under some additional regularity assumptions) that
  $(e^{tA})_{t \ge 0}$ is individually eventually positive?
\end{quote}

\paragraph*{Acknowledgements}

The authors would like to express their warmest thanks to Wolfgang
Arendt for many stimulating and helpful discussions, Anna Dall'Acqua for
her invaluable assistance concerning the bi-Laplace operator, and Khalid
Akhlil for suggesting to consider the Laplacian with non-local boundary
conditions. The first author wants to express his gratitude for a
pleasant stay at Ulm University, where part of the work was done.

\pdfbookmark[1]{\refname}{biblio}%
\providecommand{\bysame}{\leavevmode\hbox to3em{\hrulefill}\thinspace}

\end{document}